\documentclass{amsart}
\usepackage{tikz-cd}
\usepackage{amsmath}
\usepackage{amsfonts}
\usepackage{mathtools}
\usepackage{amsmath,amssymb,amsthm,mathrsfs}
\usepackage[utf8]{inputenc}
\usepackage{aeguill}
\usepackage{hyperref}
\usepackage{enumitem}

\theoremstyle{plain} 
\newtheorem{theorem}{Theorem} 
\numberwithin{theorem}{section}
\newtheorem*{theorem*}{Theorem}
\newtheorem{prop}[theorem]{Proposition}
\newtheorem{prop-def}[theorem]{Proposition-Definition}
\newtheorem{lemma}[theorem]{Lemma}
\newtheorem{coro}[theorem]{Corollary}

\newtheorem{definition}[theorem]{Definition}
\theoremstyle{definition}

\newtheorem{example}{Example}
\theoremstyle{remark} 
\newtheorem{remark}[theorem]{Remark}

\theoremstyle{definition}

\newcommand{\pa}{\partial}

\newcommand{\wt}{\widetilde}

\newcommand{\RM}{\backslash}
\newcommand{\be}{\begin{equation} }
\newcommand{\ee}{\end{equation} }

\usepackage{color}

\newcommand{\bbC}{\mathbb{C}}
\newcommand{\C}{\mathbb{C}} %is there a conflict somewhere?
\newcommand{\D}{\mathbb{D}}

\newcommand{\N}{\mathbb{N}}

\newcommand{\Q}{\mathbb{Q}}

\newcommand{\Z}{\mathbb{Z}}

\newcommand{\cH}{\mathcal{H}}

\newcommand{\cL}{\mathcal{L}}
\newcommand{\cM}{\mathcal{M}}
\newcommand{\cN}{\mathcal{N}}
\newcommand{\cO}{\mathscr{O}}

\newcommand{\cQ}{\mathcal{Q}}

\newcommand{\cS}{\mathcal{S}}
\newcommand{\cT}{\mathcal{T}}

\newcommand{\cX}{\mathcal{X}}
\newcommand{\cY}{\mathcal{Y}}
\newcommand{\cZ}{\mathcal{Z}}

\newcommand{\bee}{\mathbf{e}}

\newcommand{\bL}{\mathbf{L}}
\newcommand{\bl}{\mathbf{l}}

\newcommand{\bR}{\mathbf{R}}

\newcommand{\bs}{\mathbf{s}}

\newcommand{\gr}{\textup{gr}}

\newcommand{\grf}{\gr^F_\bullet}

\newcommand{\sL}{\mathscr{L}}
\newcommand{\sM}{\mathscr{M}}

\newcommand{\sO}{\mathscr{O}}

\newcommand{\sI}{\mathscr{I}}

\newcommand{\shTA}{\mathscr{T}}
\newcommand{\shD}{\mathscr{D}}
\newcommand{\sN}{\mathscr{N}}
\newcommand{\shM}{\mathscr{M}}

\newcommand{\Spec}{\textup{Spec }}

\newcommand{\DR}{\textup{DR}}

\newcommand{\Mod}{\textup{Mod}}

\newcommand{\Coker}{\textup{Coker}}

\newcommand{\Ker}{\textup{Ker}}

\newcommand{\End}{\textup{End}}
\newcommand{\Kos}{\textup{Kos}}

\newcommand{\supp}{\textup{supp}}

\newcommand{\ba}{{\bf a}}

\newcommand{\bk}{{\bf k}}

\newcommand{\bff}{{\bf f}}

\newcommand{\rel}{{\textup{rel}}}
\newcommand{\Ann}{{\textup{Ann}}}

\newcommand{\Ch}{\textup{Ch}}
\newcommand{\CC}{\textup{CC}}

\title{Characteristic cycles associated to holonomic $\shD$-modules}
\author{Lei Wu}
\address{Lei Wu, Department of Mathematics, KU Leuven, Celestijnenlaan 200B, B-3001 Leuven, Belgium}
\email{lei.wu@kuleuven.be}

\begin{document}


\begin{abstract}
We study relative and logarithmic characteristic cycles associated to holonomic $\shD$-modules. 
As applications, we obtain:
(1) an alternative proof of Ginsburg's log characteristic cycle formula for lattices of regular holonomic $\shD$-modules following ideas of Sabbah and Briancon-Maisonobe-Merle,
and (2) the constructibility of the log de Rham complexes for lattices of holonomic $\shD$-modules, which is a natural generalization of Kashiwara's constructibility theorem. 

\end{abstract}
\maketitle

\section{Introduction}
The characteristic variety of a coherent $\shD$-module with a good filtration is the support of the associated graded module on the cotangent bundle (see \cite{Kasbook}). Characteristic cycles can be obtained with multiplicities taken into account. They can also be considered relative to smooth morphisms (or holomorphic submersions under the analytic setting) and from a logarithmic perspective.
%Compared to characteristic varieties, characteristic cycles constain more geometric information. 
%One can also consider $\shD$-modules relative to smooth morphisms (or holomorphic submersions under the analytic setting) as well as $\shD$-modules from a logarithmic perspective. Then one gets relative/logarithmic characteristic cycles
%\footnote{Relative characteristic varieties are called characteristic varieties associated to submersions in \cite{Schpbook}}
%and log characteristic cycles respectively. 
See \S\ref{subsec:relchcc} and \S\ref{sec:logdmodule} for definitions.  

In this paper, we study characteristic cycles of relative $\shD$-modules associated to (regular) holonomic $\shD$-modules. We apply the relative characteristic cycles to studying the logarithmic characteristic cycles for lattices of regular holonomic $\shD$-modules and the constructibility of logarithmic de Rham complexes.

\subsection{Constructibility of log de Rham complexes for lattices}
For holonomic $\shD$-modules on complex manifolds, Kashiwara's constructibility theorem \cite{Kascons} says that the de Rham complexes of the holonomic modules are constructible (they are indeed perverse, see also \cite[Theorem 4.6.6]{HTT}). Our first two main results are constructibility and perversity of log de Rham complexes of lattices.  
%We first discuss the constructibility of log de Rham complexes (cf. \S\ref{subsec:logdeRhamcons}) for lattices of holonomic $\shD$-modules. 

Suppose that $(X,D)$ is a pair consisting of a complex manifold together with a reduced normal crossing divisor $D$, called an \emph{analytic smooth log pair}.
Let $\shD_{X,D}$ be the sheaf of rings of holomorphic logarithmic differential operators, that is, the sub-sheaf of $\shD_X$ consisting of differential operators preserving the defining ideal of $D$. Let $\cM$ be a coherent $\shD_X$-module. We now consider a $\shD_{X,D}$-\emph{lattice} $\bar \cM$ of $\cM$, 
a special $\shD_{X,D}$-module  associated to $\cM$ (see \S\ref{subsec:lattices} for definition). Typical examples of lattices include Deligne lattices (cf. \cite[\S4.4]{WZ}) and lattices given by the graph embedding construction of Malgrange (see  \S\ref{sec:gemMal} for details). 

\begin{theorem}\label{thm:constrlattice}
Suppose that $(X,D)$ is an analytic smooth log pair and that $\cM$ is a holonomic $\shD_X$-module. If $\bar\cM$ is a $\shD_{X,D}$-lattice of $\cM$, then the log de Rham complex $\DR_{X,D}(\bar\cM)$
is constructible.
\end{theorem}
The above theorem naturally generalizes Kashiwara's constructibility theorem and answers the question at the beginning of \cite{WZ}. The constructible complex $\DR_{X,D}(\bar\cM)$ is not perverse in general and the stratification of the constructible complex $\DR_{X,D}(\bar\cM)$ is determined by the stratification of $\Ch(\cM(*D))$ (see Remark \ref{rmk:stratificationlogDR}).
\begin{theorem}\label{thm:j*j!DR}
In the situation of Theorem \ref{thm:constrlattice}, $\DR_{X,D}(\bar\cM(kD))$ are perverse locally on a relative compact open subset of $X$ $($or globally when $X$ is algebraic$)$ for all $|k|\gg 0$. Moreover,
if $\cM$ is regular holonomic, then locally on a relative compact open subset of $X$ $($or globally when $X$ is algebraic$)$ we have natural quasi-isomorphisms
\begin{enumerate}
    \item $\DR_{X,D}(\bar\cM(kD))\stackrel{q.i.}{\simeq} Rj_*\DR(\cM|_U)$,
    \item $\DR_{X,D}(\bar\cM(-kD))\stackrel{q.i.}{\simeq} j_!\DR(\cM|_U)$
\end{enumerate}
for all integral $k\gg 0$, where $j\colon U=X\setminus D\hookrightarrow X$ is the open embedding.
\end{theorem}
Taking the lattice $\bar\cM=\sO_X$ in Theorem \ref{thm:j*j!DR} (1),  we recover the Grothendieck comparison \cite{Grocm},
\[[\sO_X\rightarrow \Omega^1_X(\log D)\rightarrow\cdots \rightarrow\Omega^n_X(\log D)]\stackrel{q.i.}{\simeq}Rj_*\C_U[n],\]
where $n$ is the dimension of $X$. See also \cite[Theorem 1.2]{WZ}. Theorem \ref{thm:j*j!DR} for lattices given by the graph embedding construction has been used in studying the cohomology support loci of rank one complex local systems \cite{BVWZ,BVWZ2}.

The Kashiwara's constructibility theorem has been extended to Riemann-Hilbert correspondence, the regular case by Kashiwara and Mebkhout independently (see for instance \cite[\S 7]{HTT}) and the irregular case by D’Agnolo and Kashiwara \cite{DAK15}. Under the logarithmic setting, Kato and Nakayama \cite{KN99} and Ogus \cite{OglogRH} studied Riemann-Hilbert correspondence for log connections on smooth log schemes, and Koppensteiner and Taplo \cite{KT} further studied the theory of logarithmic $\shD$-modules on smooth log schemes. Koppensteiner \cite{K20} (based on the work of Ogus) augmented the log de Rham complexes to graded complexes on Kato-Nakayama spaces and proved a finiteness result for logarithmic holonomic $\shD$-modules. Since smooth log pairs are smooth log schemes, lattices in this paper are special examples of log $\shD$-modules in \cite{KT}.  %In \cite{KT}, the author proposed a program to extend Riemann-Hilbert correspondence in the logarithmic category.
Therefore, one can naturally ask whether Theorem \ref{thm:constrlattice} together with Theorem \ref{thm:j*j!DR} can be enhanced to a  Riemann-Hilbert correspondence on  smooth log pairs in the logarithmic category. 
Notice that our proof of Theorem \ref{thm:constrlattice} is logarithmic in nature, since it depends on the natural stratification of the normal crossing divisor $D$, which gives evidence of the existence of the log Riemann-Hilbert correspondence. 

A similar logarithmic Riemann-Hilbert program for log holonomic modules has also been proposed in \cite{KT}. However, lattices are not logarithmic holonomic \cite[Definition 3.22]{KT} in general and hence Theorem \ref{thm:constrlattice} is different from the finiteness result in \cite{K20}. It would be interesting to further understand relations between lattices and logarithmic holonomic modules.

Our proof of Theorem \ref{thm:constrlattice} and Theorem \ref{thm:j*j!DR} depends on ``transforming" log $\shD$-modules to relative $\shD$-modules and on the study of relative characteristic cycles. Typical examples of non-trivial relative $\shD$-modules arise from the construction of the generalized Kashiwara-Malgrange filtrations. Then we discuss relative characteristic cycles associated to Kashiwara-Malgrange filtrations.

\subsection{$V$-filtrations along slopes of smooth complete intersections and their relative characteristic cycles}
%We first discuss the relative $\shD$-modules given by the generalized Kashiwara-Malgrange filtrations. 
Suppose that $X$ is a smooth algebraic variety over $\C$ (or a complex manifold) and that $Y\subseteq X$ is a smooth complete intersection of codimension $r$, that is, 
\[Y=\bigcap_{j=1}^r H_j\]
where $\sum_jH_j$ is a normal crossing divisor. Let $\shD_X$ be the sheaf of rings of algebraic (or holomorphic) differential operators. Define a $\Z^r$-filtration on $\shD_X$ along $Y$ by
\be\label{eq:kKMD}
V_{\bs}\shD_X=\bigcap_{j=1}^r V_{s_j}^{H_j}\shD_X
\ee
for every $\bs=(s_1,s_2\dots,s_r)\in \Z^r$, where $V_{\bullet}^{H_j}\shD_X$ is the Kashiwara-Malgrange filtration on $\shD_X$ along $H_j$ (see Definition \ref{def:KMalongY}). Following ideas of Sabbah \cite{Sab}, for a nondegenerate slope $L$ in the dual cone $(\Z_{\ge 0}^r)^\vee$, we define the (generalized) Kashiwara-Malgrange filtration on $\shD_X$ along $L$ by 
\[^LV_{L(\bs)}\shD_X=\sum_{L(\bs')\le L(\bs)}V_{\bs'}\shD_X.\]
This gives a $\Z$-filtration $^LV_{\bullet}\shD_X$ via the isomorphism $\Z\simeq\Z^r/L^\perp$.
%where $L^\perp$ is the dual cone of the cone generated by $L$. 
For a coherent $\shD_X$-module $\cM$, one can then define the Kashiwara-Malgrange filtration (or $V$-filtration for short) on $\cM$ along $L$ (see Definition \ref{def:fltLV}). For degenerate slopes, one can reduce to the nondegenerate ones by ignoring the unrelated $H_j$.
%(see \S \ref{subsect:spalongL} for details).   

The $V$-filtration of a coherent $\shD_X$-module $\cM$ along $L$ (if exists) contains the ``deformation'' information of $\cM$. More precisely, the Rees ring $^LR_V\shD_X$ associated to $^LV_\bullet \shD_X$ gives the sheaf of differential operators relative to the (algebraic) normal deformation of $Y$ in $X$ along $L$,
\[\varphi^L:\widetilde X^L\to \bbC,\]
where $\widetilde X^L$ is the ambient space of the deformation. Hence, the Rees module $^LR_V\cM$ associated to $^LV_\bullet\cM$ is a $\shD$-module relative to $\varphi^L$. See \S\ref{subsect:spalongL} for details.

\begin{theorem}[Sabbah]\label{thm:relccL}
Suppose that $\cM$ is a holonomic $\shD_X$-module and that $Y$ is a smooth complete intersection of codimension $r$. Then $\cM$ is specializable along every slope vector $L$ (i.e. the $V$-filtration on $\cM$ along $L$ uniquely exists). Moreover, if $\cM$ is regular holonomic and the slope vector $L$ is nondegenerate, then
$\gr^{^LV}_\bullet\cM$ gives a regular holonomic $\shD$-module on $T_YX$ and we have the following formulas for characteristic cycles,
\[\CC_{\widetilde X^L/\bbC}({ }^LR_V\cM)=\overline{q^{*}_L(\CC(\cM))}\subseteq T^*(\wt X^L/\bbC)\]
and 
\[\CC(\wt\gr^{^LV}_\bullet\cM)=\overline{q^{*}_L(\CC(M))}|_{T^*T_YX}\subseteq T^*T_YX,\]
where $q_L:T^*(\wt X^L/\bbC)\setminus T^*T_YX\simeq T^*X\times\bbC^\star\to T^*X$ is the natural projection.
\end{theorem}
%Ginsburg's proof uses the characteristic cycle formula for nearby cycles and Verdier specializations.  

One can also consider the Rees ring $R_V\shD_X$ associated to the $\Z^r$-filtration $V_\bullet\shD_X$. Similar to $^LR_V\shD_X$, $R_V\shD_X$ can be seen as the sheaf of differential operators relative to the (refined) normal deformation of $Y$ in $X$ (see  \S\ref{subsec:refnormde}),
\[\varphi: \widetilde X\to \bbC^r,\]
with the $\Z^r$-grading on $R_V\shD_X$ induced from the natural toric structure on the base $\bbC^r$. Then $\varphi^L$, as well as $^LR_V\shD_X$, is obtained from $\varphi$ and $R_V\shD_X$ respectively through the base-change,
\[\iota_L\colon\C \hookrightarrow \C^r \]
induced by the one parameter subgroup of $L$. For a $\shD_X$-module $\cM$ with a good filtration $U_\bullet\cM$ over $V_\bullet\shD_X$, the associated Rees module $R_U\cM$ is then a coherent relative $\shD$-module with respect to $\varphi$. %In general, $R_U\cM$ is just torsion-free over $\bbC^r$, but not necessarily flat over the base. 
\begin{theorem}[Sabbah]\label{thm:CCrelRU}
Suppose that $\cM$ is a regular holonomic $\shD_X$-module and that Y is a smooth complete intersection of codimension $r$. Let $U_\bullet\cM$ be a good $\Z^r$-filtration over $V_\bullet\shD_X$. Then we have
\[\CC_{\widetilde X/\bbC^r}(R_U\cM)=\overline{q^{*}(\CC(\cM))}\subseteq T^*(\wt X/\bbC^r)\]
where $q:T^*(\wt X^L/\bbC)\setminus (\prod^r_{j=1}u_j=0)\simeq T^*X\times(\bbC^\star)^r\to T^*X$ is the natural projection and $(u_1,u_2,\dots,u_r)$ are coordinates of $\bbC^r$. 
%In particular, the relative characteristic cycle is independent of good $\Z^r$-filtrations.
\end{theorem}

Theorem  \ref{thm:relccL} and \ref{thm:CCrelRU} are essentially due to Sabbah (see \cite[\S 2]{Sab2}). Their proof  relies on the study of characteristic cycles for relative $\shD$-modules (see \S\ref{sec:vfilreld}). When $L=(\underbrace{1,1,\dots,1}_r)$ (use the standard dual basis of $(\Z^r)^\vee$), the characteristic cycle formula for $\gr^{^LV}_\bullet\cM$ in Theorem \ref{thm:relccL} is due to Ginsburg (\cite[Theorem 5.8]{Gil}). But the precise characteristic cycle formulas in Theorem \ref{thm:relccL} and \ref{thm:CCrelRU} seem to be missing in the literature. It is worth mentioning that Ginsburg use the characteristic cycle formula in \cite[Theorem 5.8]{Gil} to study index theorems. It would be interesting to know whether the characteristic cycle formulas in Theorem \ref{thm:relccL} are related to index theorems or even Fukaya categories (cf. \cite{NZ}).  

\subsection{Relative Riemann-Hilbert correspondence}
We give an explicit relative Riemann-Hilbert correspondence for $^LR_V\cM$. We assume $\cM$ a regular holonomic $\shD_X$-module. The relative $\shD$-module $^LR_V\cM$ has fibers
\[ \label{eq:spdefL}
\bL i_{\alpha}^*({ }^LR_V\cM)\stackrel{q.i.}{\simeq}\cM \textup{ for $0\not=\alpha\in \bbC$ and, }\bL i_{0}^*({ }^LR_V\cM)\stackrel{q.i.}{\simeq}\wt\gr^{^LV}_\bullet\cM,
\]
where $\stackrel{q.i.}{\simeq}$ denotes quasi-isomorphism
%$i_\alpha:\{\alpha\}\hookrightarrow \bbC$ is the closed embedding
and $ i_\alpha:\widetilde X^L_\alpha=(\varphi^L)^{-1}(\alpha)\hookrightarrow \widetilde X^L$ is the closed embedding. Thus, $^LR_V\shD_X$ deforms $\cM$ into $\gr^{^LV}_\bullet\cM$.
By Theorem \ref{thm:relccL}, $^LR_V\cM$ provides an example of relative regular holonomic $\shD$-modules (see Definition \ref{def:relholo} and \S\ref{subsec:bashchangerelD}). Moreover, using Lemma \ref{lm:cmmpbdr} the relative \emph{de Rham complex}, 
\[\omega_{\widetilde X^L/\bbC}\otimes^\bL_{\shD} { }^LR_V\cM,\]
has fibers $\DR(\cM)$ for $\alpha\not=0$ and the central fiber
$\DR_{T_YX}(\wt\gr^{^LV}_\bullet\cM)$, 
where $\omega_{\wt X^L/\bbC}$ is the relative canonical sheaf of $\varphi_L$ (see \S \ref{subsect:scrm} and Remark \ref{rmk:gndlw} for explicit formulas).
But the central fiber satisfies (applying Theorem \ref{thm:holnb} and Lemma \ref{lm:Lgrnearby}) 
\[\label{eq:Lvsp}
\DR_{T_YX}(\wt\gr^{^LV}_\bullet\cM)\simeq \psi_{T_YX}(Rj^L_*p^{-1}(\DR(M))) \eqqcolon \textup{Sp}^L_{T_YX}(\DR(\cM))
\]
where $\textup{Sp}^L_{T_YX}$ is the \emph{Verdier specialization} along $L$ by definition, $p: \widetilde X^L\setminus T_YX\to X$ is the natural projection and $j^L:\widetilde X^L\setminus T_YX\hookrightarrow \widetilde X^L$ is the open embedding. Thus, the relative de Rham complex 
deforms $\DR(\cM)$ into its Verdier specialization along $L$. In particular, the relative de Rham complex gives a relative constructible complex (cf. \cite{FS17}). 
%For $L=(1,1,\dots,1)$, this ``deformation''  has been known to experts for a long time, see for instance \cite[Theorem 6.17]{Gil}. 

In general,  a relative regular Riemann-Hilbert correspondence for relative regular holonomic $\shD$-modules over curves is established in \cite{FFS19}. However, the relative holonomicity in \emph{loc. cit.} is more restricted. Motivated by the above example, one might ask whether the relative Riemann-Hilbert correspondence over curves can be extended to the case by using the relative holonomicity in Definition \ref{def:relholo}.

In contrast to $^LR_V\cM$,  $R_U\cM$ is not necessarily relatively holonomic because of the main obstruction that $R_U\cM$ is only torsion-free but not necessarily flat over $\C^r$. Consequently, one cannot ``normalize" $U_\bullet\cM$ into a $\Z^r$-indexed $V$-filtration in general. See Remark \ref{rmk:nabmkm} for more discussions. 
%But we have the following relative holonomicity result:
\begin{prop}\label{prop:flatrelhol} 
In the situation of Theorem \ref{thm:CCrelRU}, if $R_U\cM$ is flat over $\C^r$, then $R_U\cM$ is relative holonomic. 
\end{prop}

\subsection{Logarithmic characteristic cycles of lattices}
If $D=\sum_{i=1}^r H_j$ is a normal crossing divisor, then 
\[\shD_{X,D}=V_{\vec0}\shD_X\]
where the latter is defined in Eq.\eqref{eq:kKMD}. 
This means that if $U_\bullet(\cM(*D))$ is a good filtration over the $\Z^r$-filtration $V_\bullet\shD_X$, then each filtrant $U_{\bs}\cM(*D)$ is a lattice of $\cM$. This is an easy relation between log $\shD$-modules and relative $\shD$-modules. See \S\ref{subsec:logtorel} for a complicated relation between them. Our next main result is an alternative proof of Ginsburg's log characteristic cycle formula based on the complicated relation.
\begin{theorem}[Ginsburg]\label{thm:CClogC}
Suppose that $(X,D)$ is an analytic smooth log pair and that $\cM$ is a regular holonomic $\shD_X$-module. If $\bar\cM$ is a $\shD_{X,D}$-lattice of $\cM$, then 
\[\CC_{\log}(\bar \cM)=\overline{\CC}^{\log}(\cM|_{U}),\]
where $\overline{\CC}^{\log}(\cM|_U)$ denotes the closure of $\CC(\cM_U)\subseteq T^*U$ inside the logarithmic cotangent bundle $T^*(X,D)$. 
\end{theorem}
Ginsburg's proof of the above theorem under the algebraic setting in \cite[Appendix A]{Gil1} uses microlocalization of $\shD$-modules and resolution of singularities. Our proof under the analytic setting relies on converting log $\shD$-modules to relative $\shD$-modules, with some ideas due to Sabbah and Briancon-Maisonobe-Merle. 

Theorem \ref{thm:CClogC} has a long history. A characteristic variety formula of the holonomic system for $\prod_{l=1}^N(f_l+\sqrt{-1}{O})^{\lambda_l}$ first appears in \cite{KK79}. For the lattice $\shD_X[\bs]({\bf f}^\bs\cdot m)$ given by the graph embedding construction (see \S\ref{sec:gemMal}), the formula for characteristic varieties is proved by  Briancon-Maisonobe-Merle  \cite[Th\'eor\`em 2.1]{BMM} and the characteristic cycle formula in this case is obtained in \cite{BVWZ2}. Ginsburg \cite{Gil1} made it in its most general form as in Theorem \ref{thm:CClogC} under the algebraic setting. See \cite{KasBf,Gil,Gil1,BVWZ2,Mai,Maihyp} for applications of Theorem \ref{thm:CClogC}. See also \cite{Callog,Cal20} for related applications.

In a follow-up paper, Theorem  \ref{thm:CClogC} in its general form is used to obtain the conclusion that the zero loci of Bernstein-Sato ideals for regular holonomic $\shD$-modules in general are always of codimension-one \cite[Theorem 3.11]{WuRHA}. This conclusion in turn plays an important role in the establishment of the Riemann-Hilbert correspondence for Alexander complexes (see \cite[\S 3]{WuRHA}).

%Theorem \ref{thm:CClogC} has many applications including:
%\begin{enumerate}
   % \item The characteristic variety formula for $\shD_X[s]f^s$ was used to study the $b$-function of $f$ in \cite{KasBf}.
   % \item Ginsburg used the characteristic cycle formula for $\shD_X[s](f^s\cdot \cM_0)$ to study the characteristic cycles for nearby/vanishing cycles.
    %\item The characteristic variety formula for $\shD_X[\bs]\bff^\bs$ was applied to study Bernstein-Sato ideals in \cite{Mai, Maihyp}. See also \cite{BVWZ,BVWZ2,Wu20} for further applications along the same line.
    %\item Ginsburg \cite{Gil1} used the log characteristic cycle formula in its general form studying admissible modules on symmetric spaces in representation theory.
    %\item It is used in \cite{WZ} obtaining log index formulas.
%\end{enumerate}

The following example shows that regularity in both Theorem \ref{thm:CClogC} and Theorem \ref{thm:relccL} is needed.
\begin{example}
We consider $\cM=\C[t,1/t]\cdot e^{1/t}$, the algebraic irregular holonomic $\shD_X$-module generated by the function $e^{1/t}$ for $X=\C$, where $t$ is the coordinate of the complex plane $\C$. Let $H$ be the divisor $\{0\}\subseteq X$. Since 
$$t\partial_t\cdot e^{1/t}=-e^{1/t}/t,$$
$\cM$ is coherent over $\shD_{X,H}$. Hence, $\cM$ is a $\shD_{X,H}$-lattice of itself and its $V$-filtration along $H$ is the trivial filtration with $\gr^V_\bullet\cM=0.$
Moreover, since
$$(1+t^2\partial_t)\cdot e^{1/t}=0,$$
one can see that $\Ch_{\log}(\cM)$ has a component over $H$.
\end{example}

\subsection{Key ideas in proving main results}
For Theorem \ref{thm:constrlattice}, we first use direct images of log $\shD_X$-modules (see \S\ref{subsec:dimagelogD}) under the graph embedding of the defining functions of the normal crossing divisor locally to reduce to the case for lattices  $\shD_{X}[\bs](\bff^\bs\cdot \cM_0)$.  The lattices $\shD_{X}[\bs](\bff^\bs\cdot \cM_0)$ can be treated as  relative $\shD$-modules over $\C[\bs]$ with independent parameters
$$\bs=(s_1,s_2,\dots,s_r).$$
We then use the fact that $\shD_{X}[\bs](\bff^\bs\cdot \cM_0)$ is indeed relative holonomic, see Theorem \ref{thm:mairelhol}, which generalizes a relative holonomicity result of Maisonobe \cite{Mai}. We then identify the log de Rham complex of $\shD_{X}[\bs](\bff^\bs\cdot \cM_0)$ as the fiber of the relative de Rham complexes over the origin ${\mathbf 0}\in \C^r=\Spec\C[\bs]$. Finally, Theorem \ref{thm:constrlattice} follows from the relative holonomicity result and Kashiwara's constructibility theorem for complexes of $\shD$-modules with holonomic cohomologies. To prove Theorem \ref{thm:j*j!DR}, we reduce the required perversity to the flatness of the twisted lattices
\[\shD_{X}[\bs](\bff^\bs\cdot \cM_0(kD))\]
over a small (analytic) neighborhood of ${\mathbf 0}\in \C^r$ for $|k|\gg0$ by applying Sabbah's generalized $b$-functions.
%in Theorem \ref{thm:mibfs}.

The key point of the proof of 
%Theorem \ref{thm:relccL}, Theorem \ref{thm:CCrelRU} and 
Theorem \ref{thm:CClogC} is to interpret log $\shD$-modules as certain relative $\shD$-modules. 
%For Theorem \ref{thm:relccL} and Theorem \ref{thm:CCrelRU}, this is done by using an observation of Sabbah,
%\[R_V(\shD_X)\simeq \shD_{\varphi},\]
%see \S\ref{subsec:normdef} and \S\ref{subsec:refnormde}. For the log $\shD$-modules in Theorem \ref{thm:CClogC}, 
More precisely, we use what we call the log rescaled families (locally) to convert lattices to relative $\shD$-modules over the log factor. See \S\ref{subsec:logtorel} for constructions. Finally we apply a relative characteristic variety formula of Sabbah and Briancon-Maisonobe-Merle (see Lemma \ref{lm:BMM}). 

\subsection{Relations to singularities in algebraic geometry} We first clarify the Bernsten-Sato polynomials (or $b$-functions) in this paper and in the literature.

The $b$-functions in Definition \ref{def:fltLV} are the natural generalization of the $b$-functions for the usual $V$-filtrations (see \S\ref{subsec:KMVsp} and also \cite[\S III.7]{Bj}). For a holomorphic function $f$ and for the lattices  $\shD_X[s](f^s\cdot\cM_0)$ (that is, $r=1$), the associated $b$-function is the monic polynomial $b(s)\in\C[s]$ of the least degree such that
\[b(s)\cdot \frac{\shD_X[s](f^s\cdot\cM_0)}{\shD_X[s](f^{s+1}\cdot\cM_0)}=0,\]
where $\cM_0$ is an $\sO_X$-coherent submodule of a holonomic $\shD_X$-module. 
See \cite[III.2 and VI.1]{Bj} for this case. If $\cM_0=\sO_X$, the above definition gives particularly the $b$-function for $f$. For the lattice $\shD_{X}[\bs](\bff^\bs\cdot \cM_0)$ in general, since $s_i$ is identified with $-t_i\partial_{t_i}$,
Theorem \ref{thm:mibfs} gives a polynomial $b(\bs)\in\C[\bs]$ such that 
\[b(\bs)\cdot \frac{\shD_X[\bs](
\bff^\bs\cdot\cM_0)}{\shD_X[\bs](\bff^{\bs+\vec1}\cdot\cM_0)}=0,\]
with $\vec 1=(1,1,\dots,1)$. Such $b(\bs)$ are what we mean by Sabbah's generalized $b$-functions. This can be further generalized to the definition of the Bernstein-Sato ideal $B_\bff$ of $\bff$ (see for instance \cite{Budur}). 

Now we discuss the relations in between $V$-filtrations (and/or $b$-functions) and singularities in algebraic geometry. The $b$-function of $f$ provides a useful tool to study singularities of the divisor of $f$, see for instance \cite{KasBf,ELSV,BSaito,Kolpair,Ste}. For multiple functions, Budur, Musta\c{t}\v{a} and Saito \cite{BMS} defined the Bernstein-Sato polynomials (or $b$-functions) for arbitrary schemes in $X=\bbC^n$ by considering the $V$-filtration along smooth subvarieties (see Definition \ref{def:fltV}).\footnote{The $V$-filtration in \emph{loc. cit.} is the $\Q$-indexed one. One can refine the $\Z$-index to the $\Q$-index by a standard procedure using $b$-functions.} More precisely, they considered the $V$-filtration and $b$-functions along the slope $L=(1,1,\dots,1)$ for the holonomic module $\iota_{\bff+}\sO_X$, where $\bff=(f_1,\dots,f_r)$ generate the ideal of an affine scheme. They can reinterpret the log-canonical threshold as well as other jumping coefficients of the multiplier ideals of the scheme \cite{Laz} as roots of the $b$-function of the $V$-filtration (\cite[Theorem 2]{BMS}). By Theorem \ref{thm:mibfs}, one can see that if $L=(1,\dots,1)$ is not a slope of an adapted cone of certain good $\Z^r$-filtration on $\iota_{\bff+}\sO_X$, then the $b$-function of the scheme is irrelevant to the generalized $b$-function of Sabbah. 
%This clarifies the relation between the $b$-function of Budur-Musta\c{t}\v{a}-Saito and the generalized $b$-function of Sabbah (cf. \cite[\S5]{BMS}). 
In fact, this can be further refined in terms of Bernstein-Sato ideals with the help of a result of Maisonobe. More precisely, by \cite[R\'esultat 4]{Mai}, if $L=(1,\dots,1)$ is not a slope of the codimension one components of the zero locus of $B_\bff$, then the $b$-function of the scheme defined by $\bff$ is irrelevant to the Bernstein-Sato ideal of $\bff$. 

By Theorem \ref{thm:relccL}, one can now consider the $V$-filtration and the $b$-function of $\iota_{\bff+}\sO_X$ along an arbitrary slope $L$. It would be interesting to know whether there exist algebro-geometric interpretations of the jumping indices of the $V$-filtration and the roots of the $b$-function of $\iota_{\bff+}\sO_X$ along $L$. 

Musta\c{t}\v{a} and Popa \cite{MP1,MP2,MP3} defined Hodge ideals (see also \cite{Saitohi}) by considering the Hodge filtration of the $\shD$-module $\sO_X(*D)$ using mixed Hodge modules. It is then natural to ask whether there exist connections between $V$-filtrations of $\iota_{\bff+}\sO_X(*D)$ along $L$ and Hodge ideals of $D$, where $\prod_i f_i=0$ defines the divisor $D$.

\subsection{Outline}In \S\ref{sec:reldmodule}, we discuss the general theory of relative $\shD$-modules. Then we discuss log $\shD$-modules and the proof of Theorem \ref{thm:CClogC} in \S\ref{sec:logdmodule}. \S\ref{sec:vfilreld} is about the generalized $V$-filtrations and their relations with relative $\shD$-modules. Most of the results in \S\ref{sec:vfilreld} are essentially due to Sabbah. We give a down-to-earth exposition in \S\ref{sec:vfilreld}, for the reason that the beautiful theory of multi-indexed filtrations of Sabbah seems to be not widely known. Also, the construction of $V$-filtrations along arbitrary slopes seems to be missing in the literature, to the best of our knowledge.  For instance, as mentioned above only the $V$-filtration along the slope $L=(1,\dots,1)$ was studied in \cite{BMS}.
Finally, we recall the graph construction of Malgrange and prove the constructibility of log de Rham complexes in \S\ref{sec:gemMal}. %Some part of \S \ref{sec:vfilreld} is overlapped with \cite{Sab,Sab2} for two reasons:
%\begin{itemize}
%    \item Our main result (Theorem \ref{thm:constrlattice}) relies on Sabbah's generalized $b$-functions.
%    \item Because of the development of relative $\shD$-modules, arguments \cite{Sab,Sab2} can be simplified and some results there can be (slightly) generalized.
%\end{itemize}

\subsection{Convention} Throughout this paper, we discuss sheaves of modules on either algebraic or analytic spaces (or both). If the underlying space is algebraic (resp. analytic), then the sheaves of modules on it are all assumed to be algebraic (resp. analytic). But, when we discuss constructible complexes (or perverse sheaves) on a complex algebraic variety, we always use the Euclidean topology. If $f:X\to Y$ is a morphism of (algebraic or analytic) schemes, we use $f^{-1}$ and $f_*$ to denote the sheaf-theoretical inverse and direct image functors respectively.  

\subsection*{Acknowledgement}
The author thanks Peng Zhou and Nero Budur for useful discussions, Claude Sabbah for answering questions and Yajnaseni Dutta and Ruijie Yang for useful comments. 

\section{Relative $\shD$-modules}\label{sec:reldmodule}
  
\subsection{Relative characteristic cycles}\label{subsec:relchcc}
We recall the theory of $\shD$-modules under the relative setting. We mainly follow \cite[Chapter III. 1.3]{Schpbook}.
Suppose that $\varphi\colon \mathcal X\to \cS$ is a smooth morphism (i.e. $d\varphi$ is surjective everywhere) of complex smooth algebraic varieties (or complex manifolds). We write by $\shTA_{
\cX/
\cS}$ the sheaf of vector fields tangent to the leaves of $\phi$. We then have an inclusion 
\[\shTA_{\cX/\cS}\hookrightarrow \shTA_{\cX}.\]
Then the sheaf of rings of relative differential operators associated to $\varphi$ is defined to be the subalgebra 
\[\shD_\varphi=\shD_{\mathcal X/\cS}\subseteq \shD_X\] 
generated by $\shTA_{X/Y}$ and $\sO_X$. Similar to the absolute case, $\shD_{\cX/\cS}$ is a coherent and noetherian sheaf of rings. Modules over $\shD_{\mathcal X/\cS}$ are called relative $\shD$-modules over $\cS$. We also write by $\Omega^1_{\mathcal X/\mathcal S}$ the relative cotangent sheaf which is defined to be  the $\sO$-dual of $\shTA_{\mathcal X/\mathcal S}$. Since $\shD_{\cX/\cS}$ is not commutative, we have both right and left $\shD_{\cX,\cS}$-modules and the side-change operator is given by tensoring $\omega_{\cX/\cS}\coloneqq \wedge^m\Omega^1_{\mathcal X/\mathcal S}$ with its quasi-inverse by tensoring $\omega^{-1}_{\cX/\cS}$, where $m=\dim \cX-\dim \cS$.

Since $\varphi$ is smooth, we have a short exact sequence of cotangent bundles 
\[0\to \mathcal X\times_\cS T^*\cS \to T^*X\to T^*(\mathcal X/S)\to 0.\]
The filtration $F_\bullet\shD_X$ by the orders of differential operators induces on $\shD_{\cX/\cS}$ the order filtration $F_\bullet\shD_{\mathcal X/S}$. Then the associated graded sheaf of rings $\gr_\bullet^F\shD_{\mathcal X/S}$ gives the algebraic structure sheaf of $T^*(\mathcal X/S)$ by the $\sim$-functor. 
%Here ``algebraic" means that we treat fibers of $T^*(\mathcal X/S)$ as algebraic affine spaces. 
In the analytic case, $\sO_{T^*(\mathcal X/S)}$ is a faithfully flat ring extension of $\widetilde\gr_\bullet^F\shD_{\mathcal X/S}$ by GAGA.

For a coherent $\shD_{\mathcal X/\cS}$-module $\mathscr M$, a filtration $F_\bullet\sM$ over $F_\bullet\shD_{\cX/\cS}$ is called \emph{good} if $\grf\sM$ is coherent over $\grf\shD_{\cX/\cS}$. Conversely, if there exists a filtration $F_\bullet\sM$ satisfying that   $\grf\sM$ is coherent over $\grf\shD_{\cX/\cS}$, then $\sM$ is coherent over $\shD_{\cX/\cS}$. 
Good filtrations for coherent $\shD_{\mathcal X/\cS}$-modules exist locally in the analytic category and globally in the algebraic category.
We define the relative characteristic variety by 
\[\Ch_{\rel}(\mathscr M)=\Ch_{\cX/\cS}(\sM)\coloneqq \supp(\widetilde\gr^F_\bullet\sM)\subseteq T^*(\mathcal X/\cS),\]
where we use the $\sim$-functor of the affine morphism 
$$\pi: T^*(\mathcal X/\cS)\to X.$$ 
By construction, $\Ch_\rel(\sM)$ is a conic subvariety of $T^*(
\cX/\cS)$, where ``conic" means that it is invariant under the natural $\C^\star$-action induced by the grading on $\gr^F_\bullet\shD_{\cX/\cS}$. Each irreducible components of $\Ch_{\rel}(\mathscr M)$ has a multiplicity. Similar to the absolute case, $\Ch_{\rel}(\mathscr M)$ and the multiplicities are independent of good filtrations. Then the relative characteristic cycle, denoted by $\CC_{\rel}(\mathscr M)$ is the associated locally finite cycles of $\Ch_{\rel}(\mathscr M)$ with multiplicities. If $\varphi$ is an algebraic smooth morphism between smooth varieties over $\bbC$, then $\CC_{\rel}(\mathscr M)$ is an algebraic cycle inside the algebraic relative cotangent bundle $T^*(\cX/\cS)$. 

Similar to the absolute case, for a relative differential operator $P\in \shD_{\cX/\cS}$ of order $k$, we can define its principal symbol, which gives a section of homogeneous degree $k$ in $\gr^F_\bullet \shD_{\cX/\cS}$. By \cite[3.24 Definition]{Bj}, we obtain the relative Poisson bracket on  $\gr^F_\bullet \shD_{\cX/\cS}$ and hence the relative Poisson bracket on $\sO_{T^*(\cX/\cS)}$ (by faithful flatness). A subvariety of $T^*(\cX/\cS)$ is called (relative) \emph{involutive} if its radical ideal sheaf is closed under the Poisson bracket. Then by Gabber's involutive theorem (see \cite[Appendix III, 3.25 Theorem]{Bj}), we obtain:
\begin{theorem}[Gabber's Involutivity]\label{thm:Ginv}
Suppose that $\sM$ is a coherent $\shD_{\cX/\cS}$-module $($left or right$)$. Then $\Ch_\rel(
\sM)$ is $($relative$)$ involutive.
\end{theorem}

Notice that the fibers of a relative involutive subvariety $\cZ\subseteq T^*(\cX/\cS)$ are not necessarily involutive. One reason is that the intersections 
\[\cZ_s=\cZ\cap T^*(\cX/\cS)_s\subseteq T^* \cX_s\]
are not always proper intersections for $s\in \cS$. If additionally $\cZ$ is smooth over $\cS$, then one can easily check that $\cZ_s\subseteq T^*\cX_s$ is either empty or involutive.  
However, we have the following relative Bernstein inequality:
\begin{theorem}[Relative Bernstein Inequality of Maisonobe]\label{thm:rbernsteinM}
Suppose that $\sM$ is a coherent $\shD_{\cX/\cS}$-module $($left or right$)$. If $\Ch_\rel(
\sM)_s$ is not empty for $s\in\cS$, then all the irreducible components of $\Ch_\rel(
\sM)_s$ are of dimension $\ge \dim\cX-\dim\cS$.
\end{theorem}
\begin{proof}
The proof is essentially the same as that of \cite[Proposition 5]{Mai}, where the author only discussed relative $\shD$-modules over $\cS=\bbC^r$. For completeness, we sketch the proof in general. We take a smooth point of $\Ch_\rel(
\sM)$ and focus on an open neighborhood $W$ around it. By generic smoothness (or Morse-Sard Theorem for critical values in the analytic case), $\Ch_\rel(
\sM)\cap W$ is smooth over an open neighborhood $U$ of $\cS$ (shrink $W$ if necessary). Then the relative involutivity in Theorem \ref{thm:Ginv} and the relative smoothness imply that $(\Ch_\rel(
\sM)\cap W)\cap T^*\cX_s$ is involutive in $T^*\cX_s$ for $s\in\cS$. In particular, the dimension of 
$$(\Ch_\rel(
\sM)\cap W)\cap T^*\cX_s$$ 
is $\ge \dim\cX-\dim\cS$. Therefore, the required statement follows from the upper semicontinuity of the dimension of fibers of $\Ch_\rel(\sM)$.      
\end{proof}

The following lemma is the relative analogue of \cite[Proposition 2.10]{Kasbook}. We leave its proof for interested readers. See also \cite[Lemma 3.2.2]{BVWZ}. 
\begin{lemma}\label{lm:unionchrel}
If 
\[0\to \sM'\longrightarrow \sM\longrightarrow\sM''\to 0\]
is a short exact sequence of coherent $\shD_{\cX/\cS}$-modules, then 
\[\Ch_\rel(\sM)=\Ch_\rel(\sM')\cup \Ch_\rel(\sM'').\]
\end{lemma}

Following \cite{Sab2}, we define the relative holonomicity as follows. 
\begin{definition}[Relative holonomicity]\label{def:relholo}
A coherent $\shD_{\cX/\cS}$-module $\sM$ is called relative holonomic over $\cS$ (or $\sO_S$) if its characteristic variety $\Ch_\rel(
\sM)$ is relative Lagrangian, that is, the fiber $\Ch_\rel(
\sM)_s$ is either empty or a (possibly reducible) Lagrangian subvariety in $T^*\cX_s$ for every $s\in \cS$.\footnote{This relative holonomicity is slightly more general than the ones in \cite{Mai, FS17, BVWZ, BVWZ2}, where the latter requires additionally that the relative Lagrangian subvarieties are independent of $s\in \cS=\C^r$.} 
\end{definition}
Following from Lemma \ref{lm:unionchrel} and Theorem \ref{thm:rbernsteinM}, we immediately have:
\begin{coro}\label{cor:relholab}
Relative holonomicity is preserved by subquotients and extensions. In particular, the category of relative holonomic modules is abelian.  
\end{coro}

\subsection{Base change for relative $\shD$-modules}\label{subsec:bashchangerelD}
We now discuss the base change for relative $\shD_X$-modules. Suppose we have the following commutative diagram,
\be\label{diag:basechanges1}
\begin{tikzcd}
\cX_\cT\arrow[r,"\mu"]\arrow[d]& \cX\arrow[d]\\
\cT\arrow[r,"\nu"]& \cS
\end{tikzcd}
\ee
so that $\cX_\cT=\cX\times_\cS\cT$. Suppose $\sM$ is a (left) $\shD_{\cX/\cS}$-module. We consider the $\sO$-pullback through $\mu$:
\[\mu^*(\sM)=\sO_{\cX_\cT}\otimes_{\mu^{-1}\sO_{\cX}}\mu^{-1}\sM.\]
Since $\mu^*{\shD_{\cX/\cS}}=\shD_{\cX_{\cT}/\cT}$, $\mu^*(\sM)$ is naturally a relative $\shD$-module over $\cT$. We then have the derived pullback functor $\bL\mu^{*}$ for relative $\shD$-modules. When the relative $\shD$-module structure is forgotten, it is exactly the derived $\sO$-module pullback functor. One can easily see that the derived functor $\bL\mu^*$ for relative $\shD$-module preserves coherence, thanks to the identification $\mu^*{\shD_{\cX/\cS}}=\shD_{\cX_{\cT}/\cT}$ again.

We write by $i_s\colon \cX_s\hookrightarrow \cX$ the closed embedding of the fiber over $s\in \cS$. A relative holonomic $\shD_{\cX/\cS}$-module $\cM$ is \emph{regular} if $\bL i^*_s(\cM)$ is a complex of $\shD_{\cX_s}$-modules with regular holonomic cohomology sheaves for every $s\in\cS$. The author is told by C. Sabbah that it is not known whether the category of regular relative holonomic $\shD$-module is closed by taking subquotients.   

For a closed subvariety $\cZ\subseteq \cS$, we denote by $i_\cZ\colon \cX_\cZ\hookrightarrow \cX$ the closed embedding with $\cX_\cZ=\cX\times_\cS \cZ$.
\begin{lemma}\label{lm:spsmsubvar}
If $\sM$ is relative holonomic over $\cS$ and $\cZ$ is a smooth subvariety, then $\bL i_\cZ^*(\cM)$ is a complex of relative holonomic cohomology sheaves over $\cZ$. In particular, $\bL^k i_s^*\cM$ is a holonomic $\shD_{\cX_s}$-module for each $k$. 
\end{lemma}

\begin{proof}
It is obvious that $\bL^k i_\cZ^*(\sM)$ is coherent over $\shD_{\cX_\cZ/\cZ}$ and that $$\Ch_\rel(\bL^k i_\cZ^*(\sM))\subseteq \Ch_\rel(\sM)|_{\cZ}\coloneqq\Ch_\rel(\sM)\cap \cX_\cZ.$$
But $\Ch_\rel(\sM)|_{\cZ}$ is relative Lagrangian over $\cZ$. Therefore,  $\bL^k i_\cZ^*(\sM)$ is relative holonomic by Theorem \ref{thm:rbernsteinM}.
\end{proof}

\begin{prop}\label{prop:spnotorsioncc}
Suppose that $\phi:\cX\to \cS$ is smooth
with $\cH\subset\cS$ a smooth divisor, 
and that $\sM$ is a coherent $\shD_{\cX/\cS}$-module. If $\sM$ has no torsion subsheaf supported on $\cX_\cH$ and if the cycle $\CC_\rel(\sM)$ does not have components over $\cX_\cH$, then 
\[\CC_\rel(i^*_\cH\sM)=\CC_\rel(\sM)|_{\cX_\cH},\]
where $i_\cH:\cX_\cH\hookrightarrow\cX$ is the closed embedding with the fiber product $\cX_\cH=\cX\times_\cS\cH$.
\end{prop}
\begin{proof}
Since characteristic cycles are local, it is enough to assume $\cH$ defined by a regular (or holomorphic) function $h$. The torsion-free assumption implies that we have a short exact sequence 
\[0\to \sM\xrightarrow{\cdot h} \sM\rightarrow i^*_\cH(\sM)\to 0.\]
Now we pick a good filtration $F_\bullet\sM$ over $\shD_{\cX/\cS}$  and the filtration induces a filtered complex 
\[F_\bullet\sM\xrightarrow{\cdot h} F_\bullet\sM.\]
We then obtain a convergent spectral sequence with the $E^0$-page given by the following length 2 complex
\[\eta\colon\gr^F_\bullet\sM\xrightarrow{\cdot h}\gr^F_\bullet\sM.\]
Then, $i^*_\cH(\sM)$ has an induced filtration $F_\bullet(i^*_\cH(\sM))$ (good over $\shD_{\cX_{\cH}/\cH}$). By convergence of the spectral sequence (see for instance \cite[Lemme 3.5.13]{Laumon} and also \cite[3.7. Lemme]{Sab2}), we have 
\[[\gr^F_\bullet(i^*_\cH(\sM))]=[\Coker(\eta)]-[\Ker(\eta)]\]
in the Grothendieck group $K_0$. Since the characteristic cycle is well defined for $[\gr^F_\bullet(i^*_\cH(\sM))]$, the required statement for characteristic cycles follows from \cite[Appendix IV. 3.13 Proposition]{Bj}.
\end{proof}

Suppose that $\sN$ is a (left)
$\shD_{\cX_{\cT}/\cT}$-module (or more generally a complex). We consider the derived pushforward functor $\bR\mu_*(\sN)$ with $\mu$ as in Diag.\eqref{diag:basechanges1}. Since $\mu^*{\shD_{\cX/\cS}}=\shD_{\cX_{\cT}/\cT}$, we have
\[\bR\mu_*(\sN)\simeq \bR\mu_*(\mu^*(\shD_{\cX/\cS})\otimes_{\shD_{\cX_{\cT}/\cT}}\sN),\]
and hence the complex $\bR\mu_*(\sN)$ is a complex of $\shD_{\cX/\cS}$-modules by adjunction. 
\begin{prop}\label{pro:pfcohrelbasechange}
Suppose that $\nu$ is a proper morphism, $\sN$ is coherent over $\shD_{\cX_{\cT}/\cT}$ $($or more generally a complex of $\shD_{\cX_{\cT}/\cT}$-modules with coherent cohomology sheaves$)$ and $\cN$ $($or each of its cohomology sheaves$)$ admits a good filtration over $F_\bullet\shD_{\cX/\cS}$. Then $\bR^{i}\mu_*(
\sN)$ is coherent over $\shD_{\cX/\cS}$ for each $i\in\Z$.
\end{prop}
\begin{proof}
The idea of the proof of the required coherence result is similar to that of the case for absolute $\shD$-modules (see for instance \cite[Theorem 2.8.1]{Bj}). We sketch here its proof for completeness. 

By a standard procedure (see for instance the proof of \cite[Theorem 1.5.8]{Bj}), the required statement can be reduced to the case for 
\[\sN=\shD_{\cX_{\cT}/\cT}\otimes_\sO \sL\]
where $\sL$ is a coherent $\sO$-module. But this case follows immediately from the Grauert's direct image theorem for $\sO$-modules and the projection formula (since $\mu^*{\shD_{\cX/\cS}}=\shD_{\cX_{\cT}/\cT}$).
\end{proof}
It is worth mentioning that $\bR\mu_*$ does not preserve relative holonomicity under proper base changes in general (compared to Proposition \ref{prop:relpushfw}). See \S \ref{subsec:relccU} for related examples.

\subsection{Relative de Rham complexes}
We keep notations as in the previous subsection. Let $\sM$ be a (left) $\shD_{\cX/\cS}$-module. The \emph{relative de Rham complex} of $\sM$ is defined as 
\[\DR_{\cX/\cS}(\sM)\coloneqq \omega_{\cX/\cS}\otimes^\bL_{\shD_{\cX/\cS}}\sM.\]
\begin{lemma}\label{lm:cmmpbdr}
We have a natural isomorphism 
\[\bL i^*_s\circ\DR_{\cX/
\cS}\simeq \DR_{\cX_s}\circ \bL i^*_s\]
for $s\in\cS$.
\end{lemma}
\begin{proof}
The functor $\bL i^*_s$ can be rewritten as $\otimes^\bL_{\sO_\cS}\C_s$, where $\C_s$ is the residue field of $s\in\cS$. Since sections of $\sO_{\cS}$ contains in the center of $\shD_{\cX/\cS}$, the required statement follows.
\end{proof}

\subsection{Relative direct images}\label{subsec:reldim}
We discuss direct image functors for $\shD$-modules under the relative setting. We fix a morphism $f$ over $\cS$
\[
\begin{tikzcd}
\cY\arrow[rr,"f"]\arrow[dr,"\varphi_1"] &  & \cX\arrow[dl,"\varphi"] \\
&\cS
\end{tikzcd}
\]
with $\varphi_1$ and $\varphi$ smooth. We recall the definition of the relative direct image:
\[f_+(\sN)\coloneqq \bR f_*(\sM\otimes_\sO \omega_{f/\cS}\otimes_{\shD}f^*\shD_{\cX/\cS})\]
for a left $\shD_{\cY/\cS}$-module $\sN$ (or more generally a complex of left $\shD_{\cY/\cS}$-modules), where $\omega_{f/\cS}=\omega_{\cY/\cS}\otimes f^*(\omega^{-1}_{\cX/\cS})$.

The morphism $f$ over $\cS$ induces a relative Lagrangian correspondence 
\[
\begin{tikzcd}
T^*(\cY/\cS)\arrow[rd,"\phi_1"] & \cY\times _\cX T^*(\cX/\cS)\arrow[l,"\varrho_f"]\arrow[r,"\varpi_f"]\arrow[d] & T^*(\cX/\cS)\arrow[ld,"\phi"]\\
&\cS&
\end{tikzcd}
\]
See for instance \cite[\S 2.4]{HTT} for the absolute Lagrangian correspondence.

The following proposition is a generalization of \cite[Theorem 1.17(a)]{FS17}.
\begin{prop}\label{prop:relpushfw}
With above notations, let $\sN$ be a coherent $\shD_{\cY/\cS}$-module with a good filtration over $F_\bullet\shD_{\cY/\cS}$. 
\begin{enumerate}
    \item If $f$ is proper over $\cS$, then $f_+(\sN)$ is a complex of $\shD_{\cX/\cS}$-modules with coherent cohomology sheaves.
    \item If moreover $\sN$ is relative holonomic, then $f_+(\sN)$ is a complex of $\shD_{\cX/\cS}$-modules with relative holonomic cohomology sheaves.
    \end{enumerate}
\end{prop}
\begin{proof}
Since $(\sN, \C_\cY)$ gives us a good relative elliptic pair (see \cite[Definition 2.14]{SS94}), the first statement follows from Theorem 4.2 in \emph{loc. cit.} If moreover $\sN$ is relative holonomic, then by Corollary 4.3 in \emph{loc. cit.} we have 
\[\Ch_\rel(\cH^i f_+(\sN))\subseteq \varpi_f(\varrho_f^{-1}(\Ch_\rel(\sN))\]
for each $i$. Thus, for $s\in \cS$
\[\Ch_\rel(\cH^i f_+(\sN))\cap \phi^{-1}(s) \subseteq \varpi_f(\varrho_f^{-1}(\Ch_\rel(\sN))\cap\phi^{-1}(s)\subseteq\varpi_f(\varrho_f^{-1}(\Ch_\rel(\sN)\cap\phi_1^{-1}(s))).\]
By definition $\Ch_\rel(\sN)\cap\phi_1^{-1}(s)$ is Lagrangian and hence isotropic. By \cite[(4.9)]{KasBf}, $\varpi_f(\varrho_f^{-1}(\Ch_\rel(\sN)\cap\phi_1^{-1}(s))$ and hence $\Ch_\rel(\cH^i f_+(\sN))\cap \phi^{-1}(s)$ are both isotropic. Thus, $\Ch_\rel(\cH^i f_+(\sN))\cap \phi^{-1}(s)$ is Lagrangian by Theorem \ref{thm:rbernsteinM} and $\cH^i f_+(\sN)$ is relative holonomic.
\end{proof}

The following lemma is immediate by construction, and we skip its proof.
\begin{lemma}\label{lm:cmspf+}
We have a natural isomorphism 
\[\bL i^*_s\circ f_+\simeq (f_s)_+\circ \bL i^*_s,\]
where $f_s:\cY_s\to \cX_s$ is the induced morphism over $s\in \cS$.
\end{lemma}

\begin{coro}
If $f$ is (relative) proper over $\cS$, then $f_+$ preserves relative regular holonomicity. 
\end{coro}
\begin{proof}
If $f$ is proper over $\cS$, then $f_s$ is proper for $s\in\cS$. Since $(f_s)_+$ preserves regular holonomicity, the required statement follows from Lemma \ref{lm:cmspf+}. %Notice that in this case, we do not need the existence of good filtrations as in Proposition \ref{prop:relpushfw}. 
\end{proof}

\section{Logarithmic $\shD$-modules}\label{sec:logdmodule}
In this section, we recall $\shD$-modules under the logarithmic setting. Let $X$ be a complex manifold of dimension $n$ 
and let $D$ be a normal crossing divisor. We call such $(X,D)$ a (analytic) smooth log pair. We write by $\shD_{X,D}$ the subalgebra of $\shD_X$ consisting of differential operators preserving the ideal sheaf of $D$. In local coordinates $(x_1,x_2,\dots,x_n)$ on an open neighborhood $U$ with $D|_U=(\prod_{i=1}^rx_i=0)$ for some $r\le n$, $\shD_{X,D}$ is the subalgebra generated by $\sO_U$ and 
$$x_1\partial_{x_1},\dots,x_k\partial_{x_r},\partial_{x_{r+1}},\dots,\partial_{x_n}.$$
Since $D$ is normal crossing, $\shD_{X,D}$ is a coherent and noetherian sheaf of rings. Modules over $\shD_{X,D}$ are called logarithmic (log) $\shD$-modules.

The order filtration $F_\bullet \shD_X$ induces the order filtration $F_\bullet \shD_{X,D}$ such that the analytification of $\grf\shD_{X,D}$ gives us the structure sheaf of the log cotangent bundle $T^*(X,D)$. For a coherent $\shD_{X,D}$-module $\cM$, one can define the log characteristic variety
\[\Ch_{\log}(\cM)\subseteq T^*(X,D),\]
and the log characteristic cycle $\CC_{\log}(\cM)$ similar to the relative case. When $D=\emptyset$, we use $\Ch(\cM)$ (resp. $\CC(\cM)$) to denote the characteristic variety (resp. cycle) of the $\shD_X$-module $\cM$.

\subsection{Log de Rham complexes}\label{subsec:logdeRhamcons}
Suppose that $\cM$ is a left $\shD_{X,D}$-module on a smooth log pair $(X,D)$. Similar to the absolute case, $\omega_X(D)$ is a right $\shD_{X,D}$-module, where $\omega_X$ is the sheaf of the top forms on $X$. The \emph{log de Rham complex} of $\cM$ is defined as 
\[\DR_{X,D}(\cM)\coloneqq \omega_X(D)\otimes^\bL_{\shD_{X,D}}\cM. \]
By \cite[Lemma 2.3]{WZ}, we have 
\[\DR_{X,D}(\cM)\simeq [\cM\rightarrow \Omega^1(\log D)\otimes_\sO\cM\rightarrow\cdots\rightarrow \Omega^n(\log D)\otimes_\sO\cM]\]
where the complex on the right-hand side starts from the degree $-n$-term and $\Omega^i(\log D)$ denotes the sheaf of the degree $i$ log forms. 
In local coordinates 
$$(x_1,x_2,\dots,x_n) \textup{ with } D|_U=(\prod_{i=1}^rx_i=0)$$ on an open neighborhood $U$ for some $r\le n$, we further have 
\be
\DR_{X,D}(\cM)\simeq \Kos(\cM;x_1\partial_{x_1},\dots,x_r\partial_{x_r},\partial_{x_{r+1}},\dots, \partial_{x_n}),
\ee
where $\Kos$ denotes the Koszul complex of the actions $x_1\partial_{x_1},\dots,x_r\partial_{x_r},\partial_{x_{r+1}},\dots, \partial_{x_n}$ on $\cM$.

\subsection{Direct image functor}\label{subsec:dimagelogD}
Suppose that $f\colon (X,D)\rightarrow(Y,E)$ is a morphism of smooth log pairs, that is, $f\colon X\rightarrow Y$ is a morphism of complex manifolds such that $f^{-1}E\subseteq D$. Then we define the derived direct image functor for left log $\shD$-modules $\cM$ by 
\[f^{\log}_+(\cM)=Rf_*(\cM\otimes_\sO \omega_f\otimes^\bL_{\shD_{X,D}}f^*\shD_{Y,E}),\]
where $\omega^{\log}_f$ is the canonical sheaf of $f$,
\[\omega^{\log}_f=\omega_X(D)\otimes_\sO f^*(\omega^{-1}_{Y}(E)).\]
The above definition is compatible with the direct images in the relative case in \S\ref{subsec:reldim} if one take $D$ and $E$ empty.

\begin{prop}\label{prop:logpushfdr}
Let $f\colon (X,D)\rightarrow(Y,E)$ be a morphism of smooth log pairs and let $\cM$ be a $($left$)$ $\shD_{X,D}$-modules.
Then
\[Rf_*\DR_{X,D}(\cM)\simeq \DR_{Y,E}(f^{\log}_+\cM).\]
\end{prop}
\begin{proof}
The proof of this proposition is exactly the same as the non-log case (cf. \cite[Theorem 4.2.5]{HTT}). We leave the detail for interested readers. 
\end{proof}

\subsection{Lattices}\label{subsec:lattices}
We recall the definition of lattices in the analytic setting. Suppose that $(X,D)$ is an analytic smooth log pair and $\cM$ is a coherent $\shD_X$-module. We write the algebraic localization of $\cM$ along $D$ by
\[\cM(*D)=\cM\otimes_\sO \sO_X(*D),\]
where 
\[\sO_X(*D)=\lim_{k\to \infty}\sO_X(kD).\]
Notice that in general $\cM(*D)$ is not even coherent over $\shD_X$. However, if $\cM$ is holonomic, then so is $\cM(*D)$, since holonomicity is preserved under tenser products over $\sO$. A coherent $\shD_{X,D}$-submodule 
$$\bar\cM\subseteq \cM(*D)$$ is called a $\shD_{X,D}$-\emph{lattice} of $\cM(*D)$ (or $\cM$) if $\bar\cM|_{X\setminus D}=\cM|_{X\setminus D}$. By definition, lattices of $\cM$ do not have torsion subsheaves supported on $D$ (although, $\cM$ might have). The prototype examples of lattices are Deligne lattices of local systems (see \cite[\S 4.4]{WZ}).
%Its (log) characteristic cycle $\CC_{\log}(\bar\cM)$ is a cycle in $T^*(X,D)$, the (algebraic) log cotangent bundle (see \cite[\S 2]{WZ19} for definitions) and we use $\Ch_{\log}(\bar\cM)$ to denote its characteristic variety, that is, the support of $\CC_{\log}(\bar\cM)$. When $D=\emptyset$, we use $\Ch(\cM)$ (resp. $\CC(\cM)$) to denote the characteristic variety (resp. characteristic cycle) of the $\shD_X$-module $\cM$. Since the base space $X$ is analytic, the (log) characteristic cycles are only locally finite. 

\subsection{From log to relative}\label{subsec:logtorel}
In this subsection, we discuss the connection between log $\shD$-modules and relative $\shD$-modules. We focus locally on $W$ a polydisc, 
\[W=\D^k_x\times \D_t^r\textup{ with coordinates }(x_1,\dots,x_k,t_1,\dots,t_r)\]
and a divisor $D$ given by $t_1\cdot t_2\cdots t_r=0$. In particular, $(W,D)$ is a smooth log pair. We consider the log rescaled families 
\[\widetilde W_j=W\times \D_y^j\]
with $y(j)=(y_1,\dots,y_j)$ the coordinates on $\D_y^j$ for $0\le j\le r$ ($\wt W_0=W$), and the maps 
\[p_j\colon \widetilde W_j\rightarrow W, (x,t,y(j))\mapsto(x,e^{y(j)}t),\]
where we abbreviate $(x_1,\dots,x_r)$ as $x$, $(e^{y_1}t_1,\dots,e^{y_j}t_j,t_{j+1},\dots,t_r)$ as $e^{y(j)}x$ and etc. Then we have the commutative diagram for each $0\le j<k$ 
\be
\begin{tikzcd}
\wt W_j \arrow[rr,hook,"i_j"]\arrow[rd,"p_j"] & & \wt W_{j+1}\arrow[ld,"p_{j+1}"]\\
& W &
\end{tikzcd}
\ee
where the inclusion is
$$i_j:\wt W_j\hookrightarrow \wt W_{j+1}, (x,t,y(j))\mapsto (x,t,y(j),0).$$

Let $U = \D^k_x \times (\D^\star)^r_t \subset W$ with $\D^\star$ the punctured disk. Then we define $\wt U_j = p^{-1}(U)$, $\wt D_j = p^{-1}_j(D)$. 
%We will use $T_{\log}^* W = T^*(W, D)$, and $T_{\log}^*\wt W = T^*(\wt W, \wt D)$.  
Given a $\shD_{W,D}$-lattice $\bar\cM$ of a $\shD_X$-module $\cM$,
we consider the pull-backs 
$\mathscr M_j \coloneqq p^*_j 
\bar\cM$ and write $p_k=p$, $\sM=\sM_k$, $\wt W=\wt W_k$, $\wt U_k=\wt U$ and $\wt D=\wt D_k$. 
\begin{lemma}\label{lm:pblogrel}
With notations above, let $\cM$ be a coherent $\shD_X$-module. Then 
\begin{enumerate}
    \item $p_j$ is smooth $($submersive$)$ for each $j$,
    \item $\shM_j$ is a coherent $\shD_{\wt W_j,\wt D_j}$-module for each $j$,
    %since $p$ is a smooth morphism (submersion).
    %, with the bouundry $\wt D$. 
    %\item $p^{-1} \cM$ is annihilated by $x_i \pa_{x_i} - \pa_{y_i}$, since $x_i \pa_{x_i} - \pa_{y_i}$ is tangent to the fiber of $p$ for each $i$. 
    \item $\shM$ is a coherent $\shD_{\wt W/\D^r_t}$-module for the natural projection
    $$ \pi_t: \wt W \to \D^r_t.$$
    %We use $\cM_{rel}$ when we view $\wt \cM$ this way. That is, we only allow for $\pa_z$, $\pa_y$ but not $\pa_t$. 
\end{enumerate}
\end{lemma}
\begin{proof}
Part (1) is obvious. For Part (2), we pick a good filtration $F_\bullet\bar\cM$ over $F_\bullet\shD_{X,D}$. Since $p_j$ is smooth, $p_j^*(F_\bullet\bar\cM)$ is a filtration of $\sM_j$ over $F_\bullet\shD_{\wt W_j,D_j}$. By construction, we have
\[t_i\partial_{t_i}\cdot (1\otimes p_j^{-1}(m))=1\otimes p_j^{-1} (t_i\partial_{t_i}\cdot m)\]
for sections $m$ of $\bar\cM$.
Therefore, $p^*_j(\grf\cM)$ is coherent over $\grf\shD_{\wt W_j,\wt D_j}$ and hence $p_j^*(F_\bullet\bar\cM)$ is a good filtration of $\sM_j$ over $F_\bullet\shD_{\wt W_j,D_j}$. In particular, $\sM_j$ is coherent over $\shD_{\wt W_j,\wt D_j}$. For Part (3), one observes that $t_i\partial_{t_i}-\partial_{y_i}$ annihilates $1\otimes p_j^{-1}(m)$. Thus, $p^*(\grf\cM)$ is coherent over $\grf\shD_{\wt W/\D^r_t}$. In consequence, $\sM$ is coherent over $\shD_{\wt W/\D^r_t}$.
\end{proof}

\begin{theorem}\label{thm:relchrellog}
Let $\cM$ be a regular holonomic $\shD_W$-module and let $\bar\cM$ be a $\shD_{W,D}$-lattice of $\cM$. Then we have:
\begin{enumerate}
    \item If $r=1$, then $\sM$ is a relative regular holonomic $\shD_{\wt W/\D^r_x}$-module.
    \item If $\sM$ is flat over $\D^r_x$ for some $r>1$, then $\sM$ is relative holonomic over $\D^r_x$. 
\end{enumerate}

\end{theorem}
If $r\ge 2$, then $\sM$ is not necessarily relative holonomic. See Example \ref{ex:notrelhollogres} in \S \ref{sec:gemMal}.

\subsection{Proof of Theorem \ref{thm:CClogC}}
Our first goal is to prove that, the characteristic variety $\Ch_{\log} (\bar\cM)$ is the closure of   $\Ch(\cM|_U)$ in the log cotangent bundle $T_{\log}^* W$, which is equivalent to prove that $\Ch_{\log} (\bar\cM)$ has no irreducible components over $D$. Then the statement of characteristic cycles follows immediately, since multiplicities are generically defined. 
%There are quite a few characteristic varieties to consider:
%$$ \Ch_{\log}(\bar\cM), \quad \Ch_{\log}(\shM), \quad \Ch_{\rel}(\shM), \quad \Ch( \cM(*D)),\quad  \Ch(p^*\cM(*D)),\quad \Ch(\cM|_U). $$
%Our goal is to relate $\Ch_{\log}(\bar\cM)$ and $ \Ch(\cM_U) $. 

Our main tool is a technical result of Sabbah \cite{Sab2} (see also \cite{BMM}). Consider a smooth submersion $\varphi: \cX \to \cS$. Let $\cN$ be a regular hononomic $\shD_\cX$-module, $\cN_{\rel}$ be a coherent $\shD_{\cX/\cS}$-submodule of $\cN$, that generates $\cN$ as $D_\cX$-module. Suppose $\Ch(\cN) = \bigcup_{Y \in L} T_{Y}^* \cX$, with a set $L$ (locally finite) consisting of irreducible subvarieties of $\cX$. Let $L_1 \subset L$ be a subset consisting of $Y$ such that $\varphi(Y)$ contains a non-empty open subset of $\cS$. Denote by $Y^{\textup{sm}}$ the smooth locus of $Y$. By generic smoothness, we can assume $Y^{\textup{sm}}$ is smooth over $\cS$ (shrink $Y^{\textup{sm}}$ if necessary). Then we obtain the relative conormal bundle of $Y^{\textup{sm}}\hookrightarrow\cX$ over certain open subset of $\cS$. We then write by $T^*_{\varphi|_Y}(\cX/\cS)$ the closure of the (generically defined) relative conormal bundle, calling it the \emph{relative conormal space} of $\varphi|_Y$. Notice that relative conormal spaces are not necessarily relative Lagrangian.

\begin{lemma}\label{lm:BMM}\cite[3.2.Th\'eor\`em]{Sab2}, \cite[Lemme 2.2]{BMM}
With notations as above, suppose there is a non-constant holomorphic function $F: \cX \to \C$, such that $L \RM L'$ are contained in $F^{-1}(0)$, and $\cN$ are without $F$-torsion. Then, we have
$$ \Ch_{\rel}(\cN_{\rel}) = \bigcup_{Y \in L_1} T^*_{\varphi|_Y}(\cX/\cS), $$
where $T^*_{\varphi|_Y}(\cX/\cS)$ is the relative conormal space of $\varphi|_Y$.
\end{lemma}

Since characteristic cycles are local, it is enough to assume $X=W$, a polydisc as in \S\ref{subsec:logtorel}.
To apply Lemma \ref{lm:BMM}, we let $\cN = p^*(\cM(*D))$, $\cN_{\rel} = \shM$, $\cX = \wt W$, $\cS = \D^r_t$, $\varphi = \pi_t$, and $F = t_1 \cdots t_r$. Suppose we have decomposition
$$ \Ch(\cM|_U) = \bigcup_{Y \in L_1} T^*_Y W $$
for a set of closed strata $L_1$ in $U$. Then we may define a set of closed strata $\wt L_1$ in $\wt W$, by sending $Y \in L_1$ to $\wt Y = \overline{p^{-1}Y} \subset \wt W$. Since $p$ is submersive, we have
$$ \Ch(p^*\wt\cM(*D))|_{\wt U} = \bigcup_{\wt Y \in \wt L_1} T^*_{\wt Y} \wt W|_{\wt U}.   $$

\begin{coro}\label{cor:chrellogrel}
$$ \Ch_{\rel}(\shM) = \bigcup_{\wt Y \in \wt L_1} T^*_{\pi_t|_{\wt Y}}(\wt W /  \D^k_x ). $$
\end{coro}
\begin{proof} 
By Lemma \ref{lm:BMM}, we only need to look at strata of $ \Ch(p^*\cM(*D))$ that project under $\pi_x$ to an open set in $\D_x^k$, hence it suffices to consider the strata that intersect with $\pi_x^{-1}( (\D^*_x)^k) = \wt U$ (by the construction of $p$). These are labeled exactly by $\wt L_1$. 
\end{proof}

\begin{proof}[Proof of Theorem \ref{thm:relchrellog}]
By Corollary \ref{cor:chrellogrel}, $\Ch_\rel(\sM)$ does not have components over $t_1\cdots t_r=0$. For every point $\alpha\in \D^r_t$, we pick $r$ general hyperplanes passing $\alpha$. By flatness assumption, the relative holonomicity in Part (2) follows from inductively applying Proposition \ref{prop:spnotorsioncc}.  Since $\sM\subseteq \cN=p^*(\cM(*D))$, $\sM$ is torsion-free over $\D^r_t$. When $r=1$, torsion-freeness implies flatness and hence the relative holonomicity of Part (1) follows. For regularity in Part (1), one observes that
\[i_\alpha^*(\sM)\simeq p_\alpha^*(\cM|_U) \]
are regular holonomic $\textup{for } 0\not=\alpha\in \D_t$,
where the morphism $p_\alpha$ is
$$p_\alpha\colon \D_x^k\times\{\alpha\}\times \D_y\to \D_x^k\times \D_t \quad (x,\alpha,y)\mapsto (x,\alpha e^y).$$
When $\alpha=0$,
\[i_0^*(\sM)\simeq p_0^*(\cM|_{(t=0)}).\]
But 
\[[\cM|_{(t=0)}]=[\Phi_{t=0}(\cM)]\]
in the Grothendieck group $K_0$ (by \cite[Proposition 1.1.2]{Gil}). Since regularity is well-defined for objects in $K_0$, $\cM|_{(t=0)}$ is regular holonomic by Theorem \ref{thm:holnb}(1). The regularity in Part (2) is similar by induction. 
\end{proof}

%We then connects $\Ch_{\log}(\bar\cM)$ and  $\overline{\Ch(\cM|_U)}^{\log}$ through a chain of identifications.
We use $T^*_{\log}$ and $T^*_\rel$ to denote the log/relative cotangent bundle. 

\begin{prop}\label{prop:logrelsp}
We have: 
\begin{enumerate}
    \item $\Ch_{\log}(\shM) = \iota(\Ch_{\rel}(\shM))$ where
    $$ T^*_{\log} \wt W = T^*\D_x^k \times T^*_{\log} \D^r_t \times T^* \D^r_y, \quad T^*_{\rel} \wt W = T^*\D_x^k \times \D^r_t \times T^* \D^r_y $$
    $$ \iota:  T^*_{\rel} \wt W \to T^*_{\log} \wt W, \quad (x,t,y, \xi_x,\xi_y) \mapsto (x,t,y, \xi_x, \wt \xi_t = \xi_y, \xi_y),$$
    and $\wt \xi_t$ is the coefficient in front of $dt / t$.  
    \item $\tilde\iota(\Ch_{\log}(\bar\cM)) = \Ch_{\log}(\shM)|_{\{y=0\}}$, where $\tilde\iota\colon T^*_{\log} W\hookrightarrow T^*_{\log} \wt W$ is given by 
    \[(x,t,\xi_x,\wt \xi_t)\mapsto(x,t,\xi_x,\wt \xi_t, \xi_y=\wt\xi_t).\]
\end{enumerate}
\end{prop}
\begin{proof}
For Part(1), since for any section $s$ in $p^{-1}(\bar\cM)$, $t_i \pa_{t_i} - \pa_{y_i}$ annihilate $s$, hence on the level of the associated graded modules, we have $\wt \xi_{t_i} - \xi_{y_i}$ annihilate $\gr^F_\bullet(\shM)$. 

Now we prove Part(2). By Lemma \ref{lm:pblogrel}, we have
$$\grf(i_r^*(\sM))\simeq i_r^*(\grf(\sM))\simeq  p_{r-1}^*(\grf\bar\cM).$$
Meanwhile, since $\grf(\sM)$ and $\sM$ both have no $y_r$-torsion, we have 
\[\wt\supp_{i^*_r(\grf\shD_{\wt W,\wt D})}(i_r^*(\grf(\sM)))=\Ch_{\log}(\shM)|_{\{y_r=0\}}\subseteq T^*_{\log} \wt W|_{y_r=0}.\]
Since $\wt \xi_{x_r} - \xi_{y_r}$ annihilates $\gr^F_\bullet(\shM)$, we further have 
\[\wt\supp_{i^*_r(\grf\shD_{\wt W,\wt D})}(i_r^*(\grf(\sM)))=\wt\iota(\Ch_{\log}(\sM_{r-1})).\]
We then do induction backwards until we obtain Part (2).
\end{proof}

Finally, we describe the closure in log cotangent bundle. 
\begin{lemma}\label{lm:unpackagelogcl}
Let $Y \subset U$ be a closed stratum, and let $\wt Y = \overline{p^{-1}(Y)} \subset \wt W$.  Then $\tilde\iota(\overline{T_Y^*U}^{\log}) = \iota(T^*_{\pi_t\mid \wt Y} (\wt W / \D_t^r))|_{\{y=0\}}$. 
\end{lemma}
\begin{proof}
Let us unpackage the definition of $\wt Y$. For each local function $f(x,t) \in \cO_U$ that vanishes on $Y$, we consider $f(x, e^y t)$ on $\wt W$. The relative conormals with a general $t$ fixed, $T^*_{\pi_t\mid \wt Y} (\wt W / \D_t^r)$,  are generated by the relative differentials
$$ d_{/\pi_t}(p^{-1}f)= (\pa_x f) dx + t (\pa_t f) e^y dy. $$
After we apply $\iota$, we get
$$ \iota(d_{/\pi_t}(p^{-1}f)) = (\pa_x f) dx + t (\pa_t f) e^y \frac{dt}{t} +  t (\pa_t f) e^y dy. $$ 
Then we restrict to $\{y=0\}$, meaning setting $y_i=0$, and forget $\xi_{y_i}$. 
\[ \iota(d_{/\pi_t}(p^{-1}f))|_{\{ y=0\}} = (\pa_x f) dx + t (\pa_t f) \frac{dt}{t}. \]
Indeed, this gives us back the $df$ in the basis section of $T_{\log}^* W$, i.e. $dx$ and $dt/t$. 

The above argument works generically. Taking closure, the proof is done by the construction of relative conormal spaces. 
\end{proof}

Now the proof of Theorem \ref{thm:CClogC} is accomplished by combining Proposition \ref{prop:logrelsp} and Lemma \ref{lm:unpackagelogcl}.

\section{Relative $\shD$-modules and $V$-filtrations}\label{sec:vfilreld}
In this section, we discuss the Kashiwara-Malgrange filtrations for $\shD$-modules in the general sense of Sabbah by using relative $\shD$-modules. %following ideas of C. Sabbah. 
For simplicity, we focus on the algebraic category unless stated otherwise, that is, all the underlying spaces and sheaves on them are algebraic in this section. See Remark \ref{rmk:gagavfil} for the analytic case. 

\subsection{Kashiwara-Malgrange filtrations}\label{subsec:KMVsp}
%We recall the Kanshiwara-Malgrange filtrations.
\begin{definition}\label{def:KMalongY}
Suppose that $X$ is a smooth complex variety and $Y$ is a smooth subvariety of $X$ with its ideal sheaf denoted by $\sI_Y$. Then the Kashiwara-Malgrange filtration on $\shD_X$ is a $\Z$-indexed increasing filtration defined by 
\[V^Y_k\shD_X\coloneqq\{P\in \shD_X | \quad P\sI_Y^{j}\subseteq \sI_Y^{j-k} \textup{ for every } j\in \Z\}\]
where $\sI_Y^{j}=\sO_X$ if $j\le 0$.\footnote{In the literature, some authors define the Kashiwara-Malgrange filtration on $\shD_X$ as the decreasing filtration, that is, $V_Y^k\shD_X\coloneqq V^Y_{-k}\shD_X$.} In particular, if $Y=H$ is a smooth hypersurface $V^Y_0\shD_X=\shD_{X,H}$, the sheaf of logarithmic differential operators along $H$.

We then define the associated Rees ring by
\[R^Y_V\shD_X\coloneqq \bigoplus_{k\in \Z}V^Y_k\shD_X\cdot u^k\subseteq \shD_X[u,1/u],\]
where the independent variable $u$ is used to help remember the grading.
\end{definition}
%By Lemma \ref{lm:Rvrel}, one can see that $R^Y_V\shD_X$ is a coherent and noetherian sheaf of rings.
\begin{definition}\label{def:fltV}
Suppose that $\cM$ is a (left) $\shD_X$-module. A $\Z$-indexed increasing filtration $\Omega_\bullet\cM$ is \emph{compatible} with $V_\bullet\shD_X$ if
\[V^Y_k\shD_X\cdot \Omega_j\cM\subseteq \Omega_{k+j}\cM \textup{ for all } k,j\in \Z.\]
A compatible filtration $\Omega_\bullet\cM$ is a \emph{good} filtration over $V_\bullet\shD_X$ if the associated Rees module 
\[R_\Omega\cM\coloneqq \bigoplus_{k\in \Z}\Omega_k\cM\cdot u^k\subseteq \cM[u,1/u]\]
is coherent over $R^Y_V\shD_X$. A good filtration $V_\bullet\cM$ is called the Kashiwara-Malgrange filtration on $\cM$ if there exists a monic polynomial $b(s)\in\bbC[s]$ with its roots having real parts in $[0,1)$ so that 
\[b(\sum_it_i\partial_{t_i}+k)\textup{ annihilates } \gr^{V^Y}_k\cM\coloneqq V_k\cM/V_{k-1}\cM \textup{ for each } k\in\Z, \]
where $(t_1=t_2=\cdots=t_r=0)$ locally defines $Y$ and $\partial_{t_i}$ are the local vector field along the smooth divisor $(t_i=0)$.
%One can check that the above requirement is independent of choices of $t_1,t_2,\dots,t_r$. 
The monic polynomial $b(s)$ of the least degree is called the Bernstein-Sato polynomial or $b$-function of $\cM$ along $Y$.
\end{definition}
One can check that the Kashiwara-Malgrange filtration on $\cM$ is unique if exists. It is obvious that the existence of the Kashiwara-Malgrange filtration on $\cM$ guarantees that $\cM$ is coherent over $\shD_X$. A coherent $\shD_X$-module $\cM$ is called \emph{specializable} along $Y$ if the Kashiwara-Malgrange filtration exists along $Y$. Furthermore, it is called $R$-\emph{specializable} if the $b$-function $b(s)$ has roots in $R$, where $R$ is a subring of $\bbC$. %Such $b(s)$ is called the $b$-function of $\cM$ along $Y$ of $\cM$ along $Y$.

\begin{theorem}[Kashiwara]\label{thm:sphol}
If $\cM$ is holonomic over $\shD_X$, then it is specializable along every submanifold $Y\subseteq X$. 
\end{theorem}
We will give a proof of the above fundamental theorem under more general settings (cf. Theorem \ref{thm:Lsphol}).

We now recall the definition of nearby cycles and vanishing cycles along smooth hypersurfaces.
\begin{definition}%[Kashiwara]
Suppose that $H$ is a smooth hypersurface and it is defined by $(t=0)$ locally.
We assume that $\cM$ is specializable along $H$ with $V_\bullet\cM$ its Kashiwara-Malgrange filtration. Then the nearby cycle of $\cM$ along $H$ is defined by 
\[\Psi_H(\cM)\coloneqq \gr^V_0\cM\]
and the vanishing cycle is 
\[\Phi_H(\cM)\coloneqq \gr^V_{1}\cM.\]
\end{definition}
Since the morphism 
\[t:\gr^V_k\cM\longrightarrow \gr^V_{k-1}\cM\]
is an isomorphism of $\shD_H$-modules for all $k\not=0$, we then have \[\Phi_H(\cM)\simeq \gr^V_k\cM\]
for $k>0$ and 
\[\Psi_H(\cM)\simeq \gr^V_k\cM\]
for $k\le -1$.
%\begin{theorem}[Kashiwara]\label{thm:sphol}
%If $\cM$ is a holonomic $\shD_X$-module, then it is specializable along every smooth hyperplane $H\subseteq X$.
%\end{theorem}
%We will give a proof of the above fundamental theorem of     Kashiwara under more general settings.
For a perverse sheaf $K$ (with complex coefficients) on $X$, we use $\psi_t(K)$ and $\phi_t(K)$ to denote the nearby cycle and vanishing cycle of $K$ along $H=(t=0)$ respectively. Let us refer to \cite[\S 8.6]{KSbook} for their definitions.

The following theorem of Kashiwara is the Riemann-Hilbert correspondence of nearby and vanishing cycles.
\begin{theorem}\cite[Theorem 2]{KasV}\label{thm:holnb}
Suppose that $\cM$ is regular holonomic and $V_\bullet\cM$ is the Kashiwara-Malgrange filtration along a smooth hypersurface $H=(t=0)$. Then 
\begin{enumerate}
    \item $\gr^V_k\cM$ is a regular holonomic $\shD_H$-module for every $k\in \Z$,
    \item $\DR_H(\Psi_H(\cM))\simeq \psi_t(\DR_X\cM)$ and  $\DR_H(\Phi_H(\cM))\simeq \phi_t(\DR_X\cM).$
\end{enumerate}
\end{theorem}

\subsection{Algebraic normal deformation}\label{subsec:normdef}
We now recall the normal deformation algebraically; see \cite[\S 4.1]{KSbook} for the topological construction.

Suppose that $Y\subseteq X$ is a smooth subvariety with the ideal sheaf $\sI_Y$.
We algebraically define a space by 
\[\widetilde X_Y\coloneqq \Spec[\bigoplus_{k\in\Z}\sI^{-k}_Y\otimes u^k],\]
where $u$ is an independent variable giving a $\bbC^\star$-action on $\widetilde X_Y$. Then the natural inclusion $\bbC[u]\hookrightarrow \bigoplus_{k\in\Z}\sI^{-k}_Y\otimes u^k$ gives rise to a smooth family 
\[\varphi_Y\colon \widetilde X_Y\to \bbC\]
so that 
\begin{enumerate}
    \item $\varphi_Y^{-1}(u)\simeq X$ if $0\not=u\in\bbC$;
    \item $\varphi_Y^{-1}(0)=\Spec[\bigoplus_{k\in\Z^{\ge0}}\sI^{k}_Y/\sI_Y^{k+1}]\eqqcolon T_YX$, the algebraic normal bundle of $Y\subseteq X$.
    %(fibers of $T_YX$ are algebraic).
\end{enumerate}
%The smooth morphism $\varphi_Y$ is called the algebraic deformation of $X$ to $T_YX$. 
%The above construction is the algebraic realization of the topological normal deformation in \cite[\S4.1]{KSbook}. 
%We will discuss the deformation under more general settings later.

By the construction of $T_YX$ and $V^Y_\bullet\shD_X$, one easily observes 
\be\label{eq:idgrvd}
\gr_\bullet^{V^Y}\shD_X\simeq \pi_*\shD_{T_YX},
\ee
where $\pi:T_YX\to X$ is the natural affine morphism (see \cite[\S II.10]{Bj} for the analytical case). %Since $T_YX$ is algebraically defined, by GAGA-principle we have an complex manifold $T_YX^{\an}$ of $T_YX$ and its sheaf of differential operators, $\shD_{T_YX^{\an}}$. It is then obvious that $\shD_{T_YX^{\an}}$ is faithfully flat over $\shD_{T_YX}$.

Under the identification \eqref{eq:idgrvd}, $\sum t_i\partial_{t_i}$ gives a global section of $\shD_{T_YX}$ corresponding to the radial vector field of the bundle $T_YX$ (with respect to the natural $\bbC^\star$-action on $T_YX$), denoted by $\sum\overline{ t_i\partial}_{t_i}\in \gr^{V^Y}_\bullet\shD_X$. In particular, $\sum\overline{ t_i\partial}_{t_i}$ is independent of choices of $t_1,\dots,t_r$. 

Now we assume that $\cM$ is a coherent $\shD_X$-module with a good filtration $\Omega_\bullet\cM$ over $V^Y_\bullet\shD_X$. 
Since $\pi$ is affine, we get a coherent $\shD_{T_YX}$-module $\widetilde\gr^\Omega_\bullet\cM$ after applying the $\sim$-functor. We then say that $\gr^\Omega_\bullet\cM$ is holonomic (resp. regular holonomic) over $\gr^{V^Y}_\bullet\shD_X$ if $\widetilde\gr^\Omega_\bullet\cM$ is so over $\shD_{T_YX}$.

The deformation $\varphi_Y$ induces a deformation from $T^*X$ to $T^*T_YX$ as follows. We consider the relative cotangent bundle of $\varphi_Y$:
$$T^*\varphi_Y: T^*(\wt X_Y/\bbC)\to \bbC.$$
The fiber of $T^*\varphi_Y$ over $0$ is $T^*T_YX$ and $T^*(\wt X_Y/\bbC)\setminus T^*T_YX\simeq T^*X\times \bbC^\star$. 

The following lemma is essentially due to Sabbah (see \cite[Lemme 2.0.1]{Sab}). We rephrase it algebraically. 
\begin{lemma}[Sabbah]\label{lm:Rvrel}
For a smooth subvariety $Y\subseteq X$, we have a natural isomorphism 
\[R^Y_V\shD_X\simeq \shD_{\wt X_Y/\bbC}.\]
\end{lemma}
\begin{proof}
We pick a (\'etale) local coordinates $(x_1,\dots, x_{n-r},t_1,\dots, t_r)$ so that $(t_1=t_2=\cdots=t_r=0)$ defines $Y$, where $n=\dim X$ and $r$ is the codimension of $Y\subseteq X$ (cf. \cite[\S A.5]{HTT}). Since $d(t^k_i\cdot u)=kt_i^{k-1}dt\cdot u$ (taking differentials over $\bbC[u]$), we have a local decomposition of the relative cotangent sheaf
\[\Omega^1_{\wt X_Y/\bbC}=\bigoplus_{k\in \Z}\sI_Y^{-k}(\bigoplus_{i=1}^r dt_i\otimes u^{k-1}\bigoplus_{j=1}^{n-r}dx_j\otimes u^{k})\]
and hence 
\be\label{eq:decompdefXY}
\shTA_{\wt X_Y/\bbC}=\bigoplus_{k\in \Z}\sI_Y^{-k}(\bigoplus_{i=1}^r \partial_{t_i}\otimes u^{k+1}\bigoplus_{j=1}^{n-r}\partial_{x_j}\otimes u^{k}).
\ee
Since $\shTA_{\wt X_Y/\bbC}$ and $\sO_{\wt X_Y}$ generate $\shD_{\wt X_Y/\bbC}$, the required isomorphism then follows. 
\end{proof}
By Lemma \ref{lm:Rvrel}, we immediately have:
\begin{prop}\label{prop:grrelsp}
\[T^*(\wt X_Y/\bbC)=\Spec [\gr^F_\bullet R^Y_V\shD_X],\]
where $F_\bullet(R^Y_V\shD_X)$ is the order filtration for relative differential operators induced from the order filtration on $\shD_X$.
\end{prop}

\subsection{Side-change for Rees modules}\label{subsect:scrm}
We discuss side-changes for $R^Y_V\shD_X$-modules. 
%First, one notice that 
%\[R^H_V\shD_X= \shD_{\tilde X/\bbC},\]
%that is, $R_V\shD_X$ is identified with $\shD_{\tilde X/\bbC}$, the sheaf of rings of algebraic relative differential operators. To be more precise, $\shD_{\tilde X/\bbC}$ is the subalgebra of $\shD_{\tilde X_H}$ generated by the algebraic relative tangent sheaf of $\phi_H$ (since $\phi_H$ is smooth). The grading of $R_V\shD_X$ corresponding to the $\bbC^*$-action on $\tilde X_H$.
The proof of Lemma \ref{lm:Rvrel} implies that 
the relative canonical sheaf of $\varphi_Y$ is 
\[\omega_{\wt X_Y/\bbC}=\bigoplus_{k\in \Z}\omega_X\otimes_{\sO_X}\sI_Y^{-k}\otimes u^{k-r},\]
where we use $\omega_-$ to denote the canonical sheaf of the smooth variety $-$. We then immediately have the side-change operators 
\[\omega_{\wt X_Y/\bbC}\otimes_\sO(\bullet)\colon\Mod^l(R^Y_V\shD_X)\longrightarrow \Mod^r(R^Y_V\shD_X)\]
and 
\[\omega^{-1}_{\wt X_Y/\bbC}\otimes_\sO(\bullet)\colon\Mod^r(R^Y_V\shD_X)\longrightarrow \Mod^l(R^Y_V\shD_X),\]
where $\Mod^l(R^Y_V\shD_X)$ (resp. $\Mod^r(R^Y_V\shD_X)$) is the abelian category of left (resp. right) $R_V\shD_X$-modules. Similar to the absolute case, the side-change operators give an equivalence between $\Mod^l(R^Y_V\shD_X)$ and $\Mod^r(R^Y_V\shD_X)$. 

Since $\omega_{\wt X_Y/\bbC}\otimes_{\bbC[u]} \bbC_0\simeq \omega_{T_YX}$ and $R^Y_V\shD_X\otimes_{\bbC[u]} \bbC_0\simeq \gr_\bullet^{V^Y}\shD_X$ ($\bbC_0$ is the residue field of $0\in \Spec\bbC[u]$), we immediately have the following commutative diagram
\be\label{eq:sideccmd}
\begin{tikzcd}
\Mod^l(R^Y_V\shD_X) \arrow[r, "\omega_{\tilde X_Y/\bbC}\otimes\bullet"]\arrow[d, "\bullet\otimes_{\bbC[u]}\bbC_0"] & \Mod^r(R^Y_V\shD_X)\arrow[l, "\omega_{\tilde X_Y/\bbC}^{-1}\otimes\bullet"]\arrow[d, "\bullet\otimes_{\bbC[u]}\bbC_0"] \\
\Mod^l(\gr^{V^Y}_\bullet\shD_X) \arrow[r, "\omega_{T_YX}\otimes\bullet"] & \Mod^r(\gr^{V^Y}_\bullet\shD_X)\arrow[l, "\omega_{T_YX}^{-1}\otimes\bullet"]
\end{tikzcd}
\ee
where 
%$-$ represents $\tilde X_Y/\bbC$ or $T_YX$ and
$\Mod^l(\gr^{V^Y}_\bullet\shD_X)$ (resp. $\Mod^r(\gr^{V^Y}_\bullet\shD_X)$) are the abelian category of graded left (resp. right) $\gr^V_\bullet\shD_X$-modules.
%Since side-change operators for $\gr_V\shD_X$-modules are given by tensor products of graded-modules and $\gr^V_0\shD_X\simeq \shD_Y[\overline{t\partial_t}]$, we further have for each $k\in \Z$
%\be\label{eq:sideccmd1}
%\begin{tikzcd}
%\Mod^l(\gr^V_\bullet\shD_X) \arrow[r, "\omega_{T_HX}\otimes\bullet"]\arrow[d,"(\bullet)_k"] & \Mod^r(\gr^V_\bullet\shD_X) \arrow[l, "\omega_{T_HX}^{-1}\otimes\bullet"]\arrow[d,"(\bullet)_k"] \\
%\Mod^l(\gr^V_0\shD_X) \arrow[r, "\omega_{H}\otimes\bullet"] & \Mod^r(\gr^V_0\shD_X)\arrow[l, "\omega_{H}^{-1}\otimes\bullet"]
%\end{tikzcd}
%\ee
%where 
%$\Mod^l(\gr^V_\bullet\shD_X)_k$ (resp. $\Mod^r(\gr^V_\bullet\shD_X)_k $) is the category of $k$-th graded piece of modules in $\Mod^l(\gr^V_\bullet\shD_X)$ and
%the vertical morphisms are taking $k$-th graded pieces.

Furthermore, since $R_V\shD_X\otimes_{\bbC[u]}\bbC_1 \simeq \shD_X$ where $\bbC_1$ is the residue field of $1\in \Spec\bbC[u]$, we also have the following commutative diagram 
\be\label{eq:sideccmd2}
\begin{tikzcd}
\Mod^l(R_V\shD_X) \arrow[r, "\omega_{\tilde X_Y}\otimes\bullet"]\arrow[d, "\bullet\otimes_{\bbC[u]}\bbC_1"] & \Mod^r(R_V\shD_X)\arrow[l, "\omega_{\tilde X_Y}^{-1}\otimes\bullet"]\arrow[d, "\bullet\otimes_{\bbC[u]}\bbC_1"] \\
\Mod^l(\shD_X) \arrow[r, "\omega_{X}\otimes\bullet"] & \Mod^r(\shD_X).\arrow[l, "\omega_{X}^{-1}\otimes\bullet"]
\end{tikzcd}
\ee

\subsection{Characteristic varieties of nearby cycles.}
In this subsection, we calculate the characteristic cycles of nearby cycles.

Suppose that $\cM$ is specializable along a smooth hypersurface $H\subseteq X$. Then $\gr_k^V\cM$ is both a coherent $\shD_{X,H}$-module and a coherent $\shD_H$-module for every $k\in\Z$. %Recall that the algebraic localization of $\cM$ along $H$ is 
%\[\cM(*H)\coloneqq \liminf{k\to \infty}\cM\otimes_\sO \sO_X(kH).\]
%If $\cM$ is holonomic, then so is $\cM(*H)$ since 
%\[\cM(*H)=\cM\otimes_\sO\sO_X(*H)\]
%and tensoring over $\sO$ preserves holonomicity. 
\begin{lemma}\label{lm:<0lattice}
If $\cM$ is holonomic, then the Kashiwara-Malgrange filtrations satisfy
\[V_k\cM=V_{k}(\cM(*H))\]
for every $k<0$, where $\cM(*H)$ is the algebraic localization of $\cM$ along $H$. In particular, $V_k\cM$ is a $\shD_{X,H}$-lattice of $\cM(*H)$ for every $k<0$.
\end{lemma}
\begin{proof}
We consider the exact sequence 
\[0\to \cT\rightarrow \cM\rightarrow \cM(*H)\rightarrow \cQ\to 0\]
where $\cT$ is the torsion subsheaf of $\cM$ supported on $\cH$, namely $\cT=\cH^0_H(\cM)$, and $\cQ$ is the quotient module, or $\cQ=\cH^1_H(\cM)$.
Since $\cT$ and $\cQ$ are supported on $H$, by Kashiwara's equivalence (cf. \cite[Theorem 1.6.1]{HTT}) $V_k\cT$ and $V_k\cQ$ are zero for all $k<0$. The exact sequence induces another exact sequence
\[0\to V_k\cT\rightarrow V_k\cM\rightarrow V_k\cM(*H)\rightarrow V_k\cQ\to 0\]
for each $k$. We thus have obtained the required statement.

\end{proof}

The following theorem is equivalent to \cite[Theorem 5.5]{Gil}, where in \emph{loc. cit.} the nearby cycle is alternatively constructed following the algebraic approach of Beilinson and Bernstein. We give it a proof by applying Theorem \ref{thm:CClogC}.
\begin{theorem}\label{thm:ccnearby}
Suppose that $\cM$ is a regular holonomic $\shD_X$-module and that $H\subseteq X$ is a smooth hypersurface. Then 
\[\overline{\CC(\cM|_{U})}|_H\subseteq T^*(X,H)\]
is a Lagrangian cycle in $T^*H\subseteq T^*(X,H)|_H$ with $U=X\setminus H$. Furthermore, the nearby cycle $\Psi_H(\cM)$ has the characteristic cycle
\[\CC(\Psi_H(\cM))=\overline{\CC(\cM|_U)}|_H\subseteq T^*H.\]
\end{theorem}
\begin{proof}
Since characteristic cycles are local, it is enough to assume $H=(t=0)$ for some local regular (or holomorphic) function $t$.
Since $\Psi_H(\cM)=\gr^V_{k}\cM$ for $k<0$, we can focus on the short exact sequence of $V_0\shD_X=\shD_{X,H}$-modules
\[0\to V_{-1}\cM\xrightarrow{\cdot t}V_{-1}\cM\rightarrow\gr^V_{-1}\cM\to 0.\]
By Lemma \ref{lm:<0lattice}, $V_{-1}\cM$ is a $\shD_{X,H}$-lattice of $\cM$. Then, by Theorem \ref{thm:CClogC}, we have 
\[\CC_{\log}(V_{-1}\cM)=\overline{\CC(\cM|_U)}.\]
Similar to the proof of Proposition \ref{prop:spnotorsioncc}, considering the above short exact sequence, we conclude that 
\[\CC_{\log}(\gr^V_{-1}\cM)=\overline{\CC(\cM|_U)}|_H\subseteq T^*(X,H)\]
and 
\[\dim(\overline{\CC(\cM|_U)}|_H)=\dim X-1.\]
Now we pick a good filtration $F_\bullet(\gr^V_{-1}\cM)$ over $F_\bullet\shD_H$. Furthermore, we have a closed embedding 
\[T^*H\hookrightarrow T^*(X,H)\] 
defined by $\wt\xi_t=0$, where $\wt\xi_t$ is the symbol of $t\partial_t$ in $\gr^V_\bullet\shD_{X,H}$. Thus, $F_\bullet(\gr^V_{-1}\cM)$ is also good over $F_\bullet\shD_{X,H}$. Since characteristic cycles are independent of good filtrations, we therefore have 
\[\CC(\Psi_H(\cM))=\CC_{\log}(\gr^V_{-1}\cM)=\overline{\CC(\cM|_U)}|_H\subseteq T^*H.\]
\end{proof}

\subsection{Generalized Kashiwara-Malgrange filtrations}
We discuss refinements of the Kashiwara-Malgrange filtration by using Sabbah's multi-filtrations. 
%We now recall Sabbah's refinement of Kashiwara-Malgrange filtrations for smooth complete intersections in the algebraic setting. 

Suppose that $X$ is a smooth complex variety and $Y\subseteq X$ a smooth subvariety of codimension $r$ such that 
\[Y=\bigcap_{j=1}^r H_j,\]
where
$H_1, H_2,\dots,H_r$ are smooth hypersurfaces intersecting transversally (that is, the divisor $D=\sum_j H_j$ has simple normal crossings). 
We then call $Y$ a \emph{smooth complete intersection} of $H_1,\dots,H_r$. 

We use ${V}^{H_j}_\bullet \shD_X$ to denote the Kashiwara-Malgrange filtration of $\shD_X$ along $H_j$ for $j=1,\dots,r$. For $\bs=(s_1,s_2,\dots, s_r) \in \Z^r$, we set 
\[V_\bs\shD_X=\bigcap_{j=1}^r { }{V}^{H_j}_{s_j} \shD_X.\]
As the index $\bs$ varies in $\Z^r$, we get an increasing $\Z^r$-filtration of $\shD_X$ with respect to the natural partial order on $\Z^r$, denote by $V_\bullet\shD_X$. One can easily check 
$$V_{\mathbf 0}\shD_X=\shD_{X,D},$$ where the latter is the sheaf of rings of log differential operators.
We write the associated Rees ring by 
\[R_V\shD_X:= \bigoplus_{\bs\in\Z^r}V_\bs\shD_X\cdot \prod_{j=1}^ru_j^{s_i},\]
where the product $\prod_{j=1}^ru_j^{s_i}$ is used to help us remember the multi-grading of $R_V\shD_X$.

For a coherent $\shD_X$-module $\cM$, similar to Definition \ref{def:fltV}, we say that a  $\Z^r$-filtration $U_\bullet\cM$ is compatible with $V_\bullet\shD_X$ if
\[V_\bs\shD_X\cdot U_\bk\cM\subseteq U_{\bk+\bs}\cM\]
for all $\bk, \bs\in \Z^r$. Such a filtration $U_\bullet\cM$ is called good over $V_\bullet\shD_X$ if its associated Rees module $R_U\cM$ is coherent over $R_V\shD_X$.

\subsection{Refinement of normal deformation}\label{subsec:refnormde}
We keep the notations as in the previous subsection. Suppose that $Y\subseteq X$ is a smooth complete intersection of $H_1,\dots,H_r$. We denote by $\sI_{H_j}$ the ideal sheaf of $H_j$ for $j=1,2,\dots,r$. %Following ideas of Sabbah in \cite{Sab}, we now give an refinement of the construction of normal deformations in \S\ref{subsec:normdef}. 
Define
$$ \wt X \coloneqq \Spec(\bigoplus_{k_1, \cdots, k_r \in \Z} \bigotimes_j {\sI_{H_j}^{-k_j} \otimes u_j^{k_j} } ).$$
Then the natural inclusion $$\bbC[u_1,\dots,u_r]\hookrightarrow \bigoplus_{k_1, \cdots, k_r \in \Z}  \bigotimes_j {\sI_{H_j}^{-k_j} \otimes u_j^{k_j} } $$
gives rise to a smooth family 
\[\varphi\colon \widetilde X\to \bbC^r\]
so that 
\begin{enumerate}
    \item $\varphi^{-1}(u_1,\dots,u_r)\simeq X$ if $(u_1,\dots,u_r)\in(\bbC^\star)^r$;
    \item $\varphi^{-1}(\mathbf 0)= T_YX$, the algebraic normal bundle of $Y\subseteq X$.
    %(fibers of $T_YX$ are algebraic).
\end{enumerate}
The $\Z^r$-grading of $\bigoplus_{k_1, \cdots, k_r \in \Z} \prod_j {\sI_{H_j}^{-k_j} \otimes u_j^{k_j} } $ induces $(\C^\star)^k$-actions on both $\wt X$ and $T_YX$. Since $Y$ is a complete intersection, the obvious thing is that
$$T_Y X = T_{H_1} X \times_X \cdots \times_X T_{H_r} X$$
and hence $T_Y X \to Y$ is a split rank $r$ vector bundle. Moreover, the induced $(\C^\star)^r$-action on $T_YX$ is given by rescaling the fibers. 

%We have the usual map $\pi: \wt X \to \C^k$, and the obvious thing is that
%$T_Z X = T_{D_1} X \times_X \cdots \times T_{D_k} X$
%And $T_Z X \to Z$ is a split rank $k$ vector bundle, and we have canonical $(\C^*)^k$ action, by rescaling the fiber. 
Similar to Lemma \ref{lm:Rvrel}, we obtain:
\begin{lemma}\label{lm:Rvrelm}
We have a natural isomorphism 
\[R_V\shD_X\simeq \shD_{\wt X/\bbC^r}.\]
\end{lemma}
From the above lemma, we immediately conclude that $R_V\shD_X$ is a coherent and noetherian sheaf of rings.

Similar to Proposition \ref{prop:grrelsp}, we have:
\begin{prop}\label{prop:grrelspT}
\[T^*(\wt X/\bbC^r)=\Spec [\gr^F_\bullet R_V\shD_X],\]
where $F_\bullet(R^Y_V\shD_X)$ is the order filtration for relative differential operators induced from the order filtration on $\shD_X$.
\end{prop}

\begin{remark}\label{rmk:gagavfil}
In the case that $X$ is a complex manifold and $Y$ is an analytic smooth complete intersection, one can construct the complex manifold $\wt X$ similar to the topological construction in \cite[\S 4.1]{KSbook} or by using open blowups as in \cite[\S 2.1]{Sab}. Then $\shD_{\wt X/\bbC^r}$ is a faithfully flat ring extension of $R_V\shD_X$ by GAGA, or more precisely
\[\shD_{\wt X/\bbC^r}=\sO_{\wt X}\otimes_{R_V\sO_X}R_V\shD_X,\]
where 
$$R_V\sO_X=\bigoplus_{k_1, \cdots, k_r \in \Z} \bigotimes_j {\sI_{H_j}^{-k_j} \otimes u_j^{k_j}}.$$
As a consequence, all the results in this section can be extended to the analytic case.
\end{remark}

%The Rees ring $R_V(D_X)$ can be viewed as the fiberwise differential operator. 

%Hence, if we specialize to fiber $\{u_i=0, \forall i\}$, it becomes the differential operator of $T_Z X$. 

\subsection{Specializability along arbitrary slopes}\label{subsect:spalongL}
Let $L=(l_1,\dots,l_r)$ be a nonzero primitive covector in $(\Z^r_{\ge0})^\vee$. We also use $L$ to denote the ray generated by the primitive vector. We call such $L$ a slope for the smooth complete intersection $Y\subseteq X$ and that $L$ is non-degenerate if each $l_j$ is not zero and degenerate otherwise. We set 
$$Y_L\coloneqq \bigcap_{l_j\not=0} H_j.$$ 
By definition, if $L$ is non-degenerate, then $Y_L=Y$. 

Given a nondegenerate slope $L$, we have a toric embedding 
$$\iota\colon\C \hookrightarrow \C^k, \textup{ by } u \mapsto (u^{l_1}, \cdots, u^{l_k}).$$ 
%We then write 
%\[Y_L= \bigcap_{l_j\not=0}H_j,\]
%the smooth complete intersection of $\{H_j\}_{l_j\not=0}$. If each $l_j\not=0$, then $Y_L=Y$.
We can pull-back $\wt X$, to get a smooth family $\varphi_L$ in the following Cartesian diagram:
\[
\begin{tikzcd}
\wt X^L \arrow[r,hook,"\iota_L"]\arrow[d,"\varphi^L"]& \wt X\arrow[d,"\varphi"] \\
\bbC\arrow[r,hook] &\bbC^r
\end{tikzcd}
\]
This can be constructed directly, as 
$$ \wt X^L =  \Spec(\bigoplus_{k_1, \cdots, k_r \in \Z}  \bigotimes_j {\sI_{H_j}^{-k_j} \otimes (u^{l_j})^{k_j} } )$$
and the fiber over $u=0$ is
$$ (\varphi^L)^{-1}(0)=\wt X_L|_{u=0} =  \Spec[ (\bigoplus_{k_1, \cdots, k_r \in \Z}  \bigotimes_j {\sI_{H_j}^{-i_j} \otimes (u^{a_j})^{i_j} } ) \otimes_{\C[u]} \C[u]/(u) ]\simeq T_{Y}X.$$
In other words, $\wt X^L$ gives a normal deformation along the slope direction $L$. The isomorphism $\wt X^L|_{u=0}\simeq  T_{Y}X$ induces a $\C^\star$-action on $T_{Y}X$ :
\be\label{eq:Lindaction}
\lambda\cdot (y_1,\dots, y_{n-r},\xi_1,\dots,\xi_{r})=(y_1,\dots, y_{n-r},\lambda^{l_1}\cdot\xi_1,\dots,\lambda^{l_{r}}\cdot\xi_{r})
\ee
for $\lambda\in\C^\star=\C\setminus\{0\} $ and $(y_1,\dots, y_{n-r},\xi_1,\dots,\xi_{r})\in T_{Y}X$.
%, where $r_L$ is the rank of $T_{Y_L}X$ and $(l_1,l_2,\dots,l_{r_L})$ is the primitive vector by taking out all the non-zero entries in $L$.

The construction of $\wt X^L$ induces the following Cartesian diagram of relative cotangent bundles:
\be
\begin{tikzcd}
T^*(\wt X^L/\C)\arrow[r,hook]\arrow[d,"T^*\varphi^L"]& T^*(\wt X/\bbC^r)\arrow[d,"T^*\varphi"]\\
\bbC\arrow[r,hook,"\iota"]&\bbC^r.
\end{tikzcd}
\ee 

By Lemma \ref{lm:Rvrelm}, we see that $R_V\shD_X$ is flat over $\bbC[u_1,u_2,\dots,u_r]$. We then set:
\[^LR_V\shD_X:= \iota_L^*(R_V\shD_X)=\shD_{\wt X^L/\bbC}.\]
In particular, $^LR_V\shD_X$ is coherent and noetherian. If $L$ is degenerate, then one can replace $Y$ by $Y_L$ to reduce to the non-degenerate case.

\begin{remark}\label{rmk:gndlw}
Similar to $\omega_{\wt X_{Y}/\C}$ in \S\ref{subsect:scrm}, one can get the explicit formula for the relative canonical sheaves:
$$\omega_{\wt X/\C^r}=\bigoplus_{k_1, \cdots, k_r \in \Z} \omega_X\otimes_\sO(\bigotimes_j {\sI_{H_j}^{-k_j} \otimes u_j^{k_j-1} } )\textup{ and }\omega_{\wt X^L_{Y}/\C}=\iota_L^*(\omega_{\wt X/\C^r}).$$
\end{remark}

%as discussed above and obtain $\wt X_L$ and $^LR_V\shD_X$ similarly.  

By construction, we have the explicit formula for $^LR_V\shD_X$: 
\[^LR_V\shD_X= \bigoplus_{k\in\Z} { }^LV_k\shD_X\cdot u^{k}\]
where by definition
\[{ }^LV_k\shD_X\coloneqq \sum_{\bs\in \Z^r, L\cdot\bs=k} V_\bs\shD_X.\]
The graded ring $^LR_V\shD_X$ then induces an increasing $\Z$-filtration $^LV_\bullet\shD_X$ on $\shD_X$. We might call $^LV_\bullet\shD_X$ the Kashiwara-Malgrange filtration of $\shD_X$ along the slope $L$. Since $\varphi^L$ is smooth and $\wt X^L|_{u=0}\simeq T_{Y_L}X$, we have that 
\be\label{eq:grLVD}
\gr^{^LV}_\bullet\shD_X\simeq \pi_* \shD_{T_{Y_L}\shD_X}
\ee
where $\pi: T_{Y_L}X\to Y_L\hookrightarrow X$ is the composition (which is an affine morphism). The $\Z$-grading of  $\gr^{^LV}_\bullet\shD_X$ corresponds to the $\bbC^\star$-action in \eqref{eq:Lindaction}.

The $\bbC^\star$-action in \eqref{eq:Lindaction} induces a radial vector field on $T_{Y_L}X$, denoted by $v_L$. We assume that locally $H_j$ are defined by $t_j=0$ where $t_j$ are among some local coordinate system $(x_1,\dots,x_{n-r},t_1,\dots,t_r)$. 
Then locally 
\[v_L=L\cdot (t_1\partial_{t_1},t_2\partial_{t_2},\dots,t_r\partial_{t_r}).\]

%Suppose that $\cM$ is a coherent $\shD_X$-module. We then can define good filtrations of $\cM$ over $^LV_\bullet\shD_X$.
\begin{definition}\label{def:fltLV}
Suppose that $Y\subseteq X$ is a smooth complete intersection of $H_1,\dots,H_r$, $\cM$ is a $($left$)$ $\shD_X$-module and $L$ is a slope. A $\Z$-indexed increasing filtration $\Omega_\bullet\cM$ is \emph{compatible} with $^LV_\bullet\shD_X$ if
\[^LV_k\shD_X\cdot \Omega_j\cM\subseteq \Omega_{k+j}\cM \textup{ for all } k,j\in \Z.\]
A compatible filtration $\Omega_\bullet\cM$ is a \emph{good} filtration over $^LV_\bullet\shD_X$ if the associated Rees module 
\[^LR_\Omega\cM\coloneqq \bigoplus_{k\in \Z}\Omega_k\cM\cdot u^k\subseteq \cM[u,1/u]\]
is coherent over $^LR_V\shD_X$. A good filtration $^LV_\bullet\cM$ is called the Kashiwara-Malgrange filtration on $\cM$ along the slope $L$ if there exists a monic polynomial $b(s)\in\bbC[s]$ with its roots having real parts in $[0,1)$ so that 
\[b(v_L+k)\textup{ annihilates } \gr^{^LV}_k\cM\coloneqq {}^LV_k\cM/{}^LV_{k-1}\cM \textup{ for each } k\in\Z. \]
The monic polynomial $b(s)$ of the least degree is called the Bernstein-Sato polynomial or $b$-function of $\cM$ along $L$.
%where $(t_1=t_2=\cdots=t_r=0)$ locally defines $Y$ and $\partial_{t_i}$ are the local vector field along the smooth divisor $(t_i=0)$. One can easily check that the above requirement is independent of choices of $t_1,t_2,\dots,t_r$. 
\end{definition}
Kashiwara-Malgrange filtrations along $L$ are unique if exist. For an arbitrary slope $L$, a coherent $\shD_X$-module $\cM$ is called $L$-\emph{specializable} if the Kashiwara-Malgrange filtration of $\cM$ along $L$ exists.  When $L=(1,1,\dots,1)$, the Definition \ref{def:fltLV} coincides with Definition \ref{def:fltV}. When $L=\bee_j$, the $j$-th unit vector in $\Z^r_{\ge0}$, it coincides with Definition \ref{def:fltV} for the case when $Y=H_j$.
In particular, specializability along $L$ is compatible with the specializability defined in \S\ref{subsec:KMVsp}. We can then similarly define $R$-specializability along $L$ for $R$ a subring of $\bbC$.

Suppose that $U_\bullet\cM$ is a good filtration of $\cM$ over $V_\bullet\shD_X$. Then for a slope $L$ we can obtain a compatible filtration $^LU_\bullet\cM$ over $^LV_\bullet\shD_X$ defined by 
\be\label{eq:filsl}
^LU_k\cM\coloneqq \sum_{\bs\in \Z^r, L\cdot\bs=k} U_\bs\cM.
\ee
We denote the associated Rees module by 
\[
^LR_U\cM\coloneqq\bigoplus_{k\in\Z}^LU_k\cM\cdot u^k.
\]
Since $^LR_U\cM\simeq \iota_L^*(R_U\cM)/\cT_u$ and the pullback functor for relative $\shD$-modules preserves coherence, we have that $^LU_\bullet\cM$ is good over $^LV_\bullet\shD_X$, where $\cT_u$ is the $u$-torsion subsheaf.

The following result is a natural generalization of Theorem \ref{thm:sphol}, which is first observed by Sabbah \cite[\S3.1]{Sab}. We provide an alternative proof with the idea essentially due to Bj\" ork. 
\begin{theorem}\label{thm:Lsphol}
Suppose that $\cM$ is a holonomic $\shD_X$-module and $L$ is a slope. Then $\cM$ is splecializable along $L$. Moreover, if $\Omega_\bullet\cM$ is a good filtration over ${}^L V_\bullet\shD_X$, then $\gr^{\Omega}_\bullet\cM$ is holonomic over $\gr^{^LV}_\bullet\shD_X$.
\end{theorem}
\begin{proof}
We take a good filtration $\Omega_\bullet\cM$ over $^LV_\bullet\shD_X$ locally. Such a filtration exists at least locally by coherence. Then we apply \cite[Appendix IV. Theorem 4.10]{Bj} and conclude that $j_{\gr^{^LV}_\bullet\shD_X}(\gr^{\Omega}_\bullet\cM)=j(\cM),$
where the first grade number is the graded number of $\gr^{\Omega}_\bullet\cM$ over $\gr^{^LV}_\bullet\shD_X$. 
%(cf. the definition in Eq. \eqref{eq:defj} and also \cite[Appendix IV. Definition 1.8]{Bj}).
Since $\cM$ is holonomic, by \cite[Appendix IV. Proposition 3.5(2)]{Bj} $j(\cM)=n$, the dimension of $X$. Hence $j_{\gr^{^LV}_\bullet\shD_X}(\gr^{\Omega}_\bullet\cM)=n$ and hence $\gr^{\Omega}_\bullet\cM$ is holonomic over $\gr^{^LV}_\bullet\shD_X\simeq\pi_*\shD_{T_{Y_L}X}$ (since $\dim T_{Y_L}X=\dim X=n$).

Now we consider the operator 
\[\theta_L\coloneqq \bigoplus_{k\in\Z}v_L+k\]
on $\gr^{\Omega}_\bullet\cM$. By construction, one can easily check 
\[\theta_L\in \End_{\gr^{^LV}_\bullet\shD_X}(\gr^{\Omega}_\bullet\cM).\]
Since $\gr^{\Omega}_\bullet\cM$ is holonomic over $\gr^{^LV}_\bullet\shD_X$, we conclude that $\theta_L$ admits a minimal polynomial $b(s)\in\bbC[\bs]$. The real parts of roots of $b(s)$ might not be contained in $[0,1)$. Namely, the good filtration $\Omega_\bullet\cM$ is not  the Kashiwara-Malgrange filtration along $L$. We then apply the procedure in the proof of \cite[Theorem 1(1)]{KasV} to adjust the roots of $b(s)$ and the filtration $\Omega_\bullet\cM$. The output of the procedure gives us the Kashiwara-Malgrange filtration $^LV_\bullet\cM$. By uniqueness, the local construction glues to the global $^LV_\bullet\cM$. Therefore, $\cM$ is specializable along $L$. 
\end{proof}

\subsection{Micro nearby cycles along arbitrary slopes}\label{subsec:mncsl}
We keep notations as in the previous subsection and continue to assume $Y\subseteq X$ a smooth complete intersection and $\cM$ a holonomic $\shD_X$-module.

By Theorem \ref{thm:Lsphol}, the Kashiwara-Malgrange filtration $^LV_\bullet\cM$ of $\cM$ along every slope $L$. For a nondegenerate slope $L$, the module $p^*_L(\cM)$ on $\wt X^L\setminus T_YX$ gives rise to a holonomic $\shD_{\wt X^L}$-module, denoted by $\wt \cM_L$ (that is, $\wt \cM_L=j^L_*(\cM_L)$), where $p_L\colon \wt X^L\setminus T_YX\simeq X\times\bbC^\star\to X$ is the natural projection and $j^L\colon \wt X^L\setminus T_YX\hookrightarrow \wt X^L$ is the open embedding. 

\begin{lemma}\label{lm:Lgrnearby}
Suppose that $\cM$ is specializable along a slope $L$. Then
\[\wt\gr^{^LV}_\bullet\cM\simeq \Psi_{u=0}(\wt \cM_L).\]
\end{lemma}
\begin{proof}
The $\C^\star$-action on $\wt X^L$ is induced by the grading of $u$ by construction. Hence, the $\C^\star$-action induces the action of the Euler vector field along the smooth divisor $T_YX=(u=0)\subseteq \wt X^L$. But the $\C^\star$-actions on $\wt X^L$ and hence $T_{Y_L}X$ are both induced by the operator 
\[\bigoplus_{k\in\Z} (L\cdot (t_1\partial_{t_1},t_2\partial_{t_2},\dots,t_r\partial_{t_r})+k).\]
The required statement then follows by definition. See also \cite[\S1.3]{BMS} for the case $L=(1,1,\dots,1)$. 
\end{proof}

By the above lemma and Theorem \ref{thm:holnb}, we obtain:
\begin{coro}\label{cor:grlhol}
If $\cM$ is regular holonomic, then so is $\gr^{^LV}_\bullet\cM$.  
\end{coro}

For a holonomic $\shD_X$-module $\cM$, we denote
\[^LR\Psi_{T_YX}(\cM)\coloneqq \wt\gr^{^LV}_\bullet\cM.\]
Following terminology in \cite{Schpbook} (and motivated by Lemma \ref{lm:Lgrnearby}), we call $^LR\Psi_{T_YX}(\cM)$ the \emph{micro nearby cycle} of $\cM$ along $L$.
\begin{theorem}\label{thm:micccl}
 Assume that $\cM$ is a regular holonomic $\shD_X$-module and $L$ is a nondegenerate slope. Then  
\[\CC_{\tilde X_H/\bbC}({}^LR_V\cM)=\overline{q^{*}_L(\CC(\cM))} \textup{ and }\CC({ }^LR\Psi_{T_YX}(\cM))=\overline{q^{*}_L(\CC(\cM))}|_{T^*T_HX}.\]
where $q_L:T^*(\wt X^L/\bbC)\setminus T^*T_YX\simeq T^*X\times\bbC^\star\to T^*X$ is the natural projection.
\end{theorem}
\begin{proof}
We apply Lemma \ref{lm:BMM} and let $
\cN_\rel= { }^LR_V\cM$, $\cN=\wt \cM_L$ and $F=u$. We thus have 
\[\CC_{\tilde X_H/\bbC}({}^LR_V\cM)=\overline{q^{*}_L(\CC(\cM))}.\]
To obtain 
\[\CC({ }^LR\Psi_{T_YX}(\cM))=\overline{q^{*}_L(\CC(\cM))}|_{T^*T_HX},\]
we apply Proposition \ref{prop:spnotorsioncc}.
\end{proof}
Theorem \ref{thm:relccL} follows from combining Theorem \ref{thm:Lsphol}, Corollary \ref{cor:grlhol} and Theorem \ref{thm:micccl}.

\subsection{Sabbah's toric base-change of $R_U\cM$}
We now recall Sabbah's toric base-changes of Rees modules. %All the results and definitions mentioned in this subsection are due to Sabbah. 

Let us first introduce some notations. We set 
\[M=\Z^r, M^+=(\Z_{\ge0})^r, M_\Q=M\times\Q \textup{ and } M^+_\Q=(\Q_{\ge 0})^r.\]
Let $N$ be the dual lattice of $M$ and define $N^+$, $N_\Q$ and $N^+_\Q$ similarly. 
%In this way, $M$ is naturally isomorphic to the index set $\Z^r$ of the $\Z^r$-filtration $V_\bullet\shD_X$.
We then fix a simplicial fan $\Sigma$ in $N_\Q$. We assume that $\Sigma$ is given by a subdivision of $N^+_\Q$. Then $\Sigma$ gives a toric variety $\cS_\Sigma$ (by making further subdivision, one can assume $\cS_\Sigma$ is smooth) and a projective birational morphism 
\[\nu_\Sigma\colon\cS_\Sigma\rightarrow\bbC^r\simeq \Spec \bbC[M^+].\]
For a cone $\Gamma\in \Sigma$, we set 
\[\check{\Gamma}\coloneqq \{m\in M \mid \langle m, n\rangle\ge 0, \forall n\in\Gamma\}.\]
Then 
\[\cS_{\Gamma}\coloneqq \Spec \bbC[\check{\Gamma}]\hookrightarrow \cS_\Sigma\]
gives an open affine patch of $\cS_\Sigma$. We denote a partial ordering induced by $\Gamma$ on $M$ by 
\[\bs\le_{\Gamma}\bs' \Leftrightarrow \bs'-\bs\in \check{\Gamma}\]
and we say 
\[\bs<_{\Gamma}\bs' \Leftrightarrow \bs\le_{\Gamma}\bs' \textup{ but not } \bs'\le_{\Gamma}\bs.\]
We use $\cL(\Gamma)$ to denote the finite set of all the primitive generators of $\Gamma$ and write
\[\cL(\Sigma)\coloneqq \bigcup_{\Gamma\in \Sigma}\cL(\Gamma).\]

We suppose that $Y\subseteq X$ is a smooth complete intersection of $H_1,\dots,H_r$ and $\cM$ a $\shD_X$-module with a $\Z^r$-filtration $U_\bullet\cM$ that is good over $V_\bullet\shD_X$. It is now natural to make $U_\bullet\cM$ and $V_\bullet\shD_X$ indexed by $M\simeq \Z^r$. 
We consider the following fiber-product diagram
\[
\begin{tikzcd}
\wt X_\Sigma \coloneqq \wt X\times_{\bbC^r} \cS_\Sigma \arrow[r,"\mu_\Sigma"]\arrow[d,"\varphi_\Sigma"] &\wt X\arrow[d, "\varphi"]\\
\cS_\Sigma\arrow[r,"\nu_\Sigma"]&\bbC^r.
\end{tikzcd}
\]
Using the flattening theorem relative over toric varieties, Sabbah and Castro proved the following fundamental theorem:
\begin{theorem}\cite[A.1.1]{Sab}\label{thm:existadfan}
Suppose that $U_\bullet\cM$ is a good multi-filtration over $V_\bullet\shD_X$ along the smooth complete intersection $Y\subseteq X$. Then there exists a simplicial fan $\Sigma$ subdividing $\N^+_\Q$ such that $\widetilde{R_U\cM}\coloneqq\mu^*_\Sigma R_U\cM/\mathcal T$ is flat over $\cS_\Sigma$, where $\mathcal T$ is the torsion subsheaf of $\mu^*_\Sigma R_U\cM$ supported over the exceptional locus of $\nu_\Sigma$.
\end{theorem}

Such $\Sigma$ in the above theorem is called a fan  \emph{adapted} to $U_\bullet\cM$. 
By construction, we have a natural inclusion of $\Z^r$-graded modules 
\[R_U\cM\hookrightarrow \cM[u_1,1/u_1,\dots,u_r,1/u_r]\]
or equivalently 
\[R_U\cM|_{\widetilde X_N}=\cM[u_1,1/u_1,\dots,u_r,1/u_r].\]
%since $R_U\cM$ can be seen as a relative $\shD$-module over the toric variety $\bbC^r$,
%where $T_N\simeq (\bbC^*)^r\hookrightarrow \bbC^r$ is the torus. 
If we have a cone $\Gamma\in \Sigma$, then we know
\[\mu^*_\Sigma R_U\cM|_{\wt X_\Gamma}=\mu^*_\Gamma R_U\cM= \bbC[\check{\Gamma}]\otimes_{\bbC[M^+]}R_U\cM.\]
Since $\mu_\Sigma$ and $\mu_{\Gamma}$ are identical over $S_N$, using the pullback functor we have an induced natural morphism 
\be \label{eq:pbmorphism}
\bbC[\check{\Gamma}]\otimes_{\bbC[M^+]}R_U\cM\rightarrow \cM[u_1,1/u_1,\dots,u_r,1/u_r].
\ee
The above morphism is neither necessarily surjective nor injective. What is the its image? To answer this question, Sabbah introduced a refined filtration for each cone $\Gamma\in\Sigma$ by
\[^\Gamma U_\bs\cM=\sum_{\bs'\le_{\Gamma}\bs}U_{\bs'}\cM, \quad\forall \bs\in M.\]
If $\Gamma$ is a unimodular cone of dimesion $<k$, then one easily sees that $^\Gamma U_\bs\cM$ only depends on the image of $\bs$ in $M/\Gamma^\perp$. Hence, in the special case that $L$ is a ray in $\Sigma$, we have 
\[^LU_\bs\cM={ }^LU_{L(\bs)}\cM,\]
that is, $^LU_\bullet\cM$ is indexed by $\Z\simeq M/\Gamma^{\perp}$. Therefore, $^LU_\bullet\cM$ coincides with the $\Z$-indexed filtration defined by \eqref{eq:filsl}.

We write the associated Rees module by 
\[^\Gamma R_U\cM\coloneqq \bigoplus_{s\in M}{}^\Gamma U_\bs \cM\cdot\prod_{j=1}^r u_j^{s_j}.\]
One observes that $^\Gamma R_U\cM$ is the image of the natural morphism \eqref{eq:pbmorphism} and that its kernel is the torsion subsheaf $\cT|_{\widetilde X_\Gamma}$. 
Therefore, we have proved that 
\[\widetilde{R_U\cM}|_{\widetilde X_\Gamma}\simeq { }^\Gamma R_U\cM.\]
Furthermore, Sabbah proved:
\begin{lemma}\cite[2.2.2.Lemme]{Sab}\label{lm:flatgamma}
If ${ }^\Gamma R_U\cM$ is flat over $\bbC[\check{\Gamma}]$, then 
\[^\Gamma U_\bs\cM=\bigcap_{L\in \cL(\Gamma)}{}^LU_{L(\bs)}\cM.\]
\end{lemma}
It is obvious that $\bbC[\check{\Gamma}]\otimes_{\bbC[M^+]}R_U\cM$ is coherent over $^\Gamma R_V\shD_X$ and hence $^\Gamma R_U\cM$ is coherent over $^\Gamma R_V\shD_X$. However, it is not always the case that $^\Gamma R_U\cM$ is coherent over $R_V\shD_X$ and hence $^\Gamma U_\bullet\cM$ is not necessarily a good filtration over $V_\bullet\shD_X$. To fix this, Sabbah defined the \emph{saturation} of $U_\bullet\cM$ by 
\[\bar U_\bs \cM\coloneqq \bigcap_{\textup{primitive vectors } L\in N^+} {}^LU_{L(\bs)}\cM.\]
\begin{theorem}[Sabbah]\label{thm:goodsat}
Suppose that $U_\bullet\cM$ is a good multi-filtration over $V_\bullet\shD_X$ and $\Sigma$ is a simplicial fan adapted to $U_\bullet\cM$. Then $\bar U_\bullet\cM$ is good over $V_\bullet\shD_X$ and 
\[\bar U_\bs \cM=\bigcap_{L\in \cL(\Sigma)}{}^LU_{L(\bs)}\cM.\]
\end{theorem}
\begin{proof}
We take a cone $\Gamma\in \Sigma$ and consider the natural surjection
\[\mu^*_\Gamma R_U\cM\rightarrow \widetilde{R_U\cM}|_{\widetilde X_\Gamma}\simeq { }^\Gamma R_U\cM.\]
By Theorem \ref{thm:existadfan}, $\widetilde{R_U\cM}$ is flat over $\cS_\Sigma$.
By Lemma \ref{lm:flatgamma}, we hence know 
\[\widetilde{R_U\cM}|_{\wt X_\Gamma}=\bigoplus_{s\in M}(\bigcap_{L\in \cL(\Gamma)}{}^LU_{L(\bs)}\cM)\cdot\prod_{j=1}^r u_j^{s_j}.\]
Since $\{\wt X_\Gamma\}_{\Gamma\in \Sigma}$ gives a covering of $\wt X_\Sigma$, we hence obtain that 
\[\mu_{\Sigma*}(\widetilde{R_U\cM})=\bigoplus_{s\in M}(\bigcap_{L\in \cL(\Sigma)}{}^LU_{L(\bs)}\cM)\cdot\prod_{j=1}^r u_j^{s_j}.\]
Since $\nu_\Sigma$ is projective, by Proposition \ref{pro:pfcohrelbasechange} we conclude that $\mu_{\Sigma*}(\widetilde{R_U\cM})$ is coherent over $R_V\shD_X$.

By construction, we know that 
\[^\Gamma U_\bs \cM \subseteq \bigcap_{\textup{primitive vectors }L\in\Gamma\cap N^+}{}^LU_{L(\bs)}\cM\]
where on the right hand side the intersection is over all primitive vectors in $\Gamma\cap N^+$ (not just generators of $\Gamma$)
and hence 
\[^\Gamma U_\bs \cM =\bigcap_{L\in \cL(\Gamma)}{}^LU_{L(\bs)}\cM= \bigcap_{\textup{primitive vectors }L\in\Gamma\cap N^+}{}^LU_{L(\bs)}\cM\]
thanks to Lemma \ref{lm:flatgamma} again. Therefore,
\[\bar U_\bs\cM= \bigcap_{\Gamma\in\Sigma}{}^\Gamma U_\bs \cM=\bigcap_{L\in \cL(\Sigma)}{}^LU_{L(\bs)}\cM.\]
and 
\[\mu_{\Sigma*}(\widetilde{R_U\cM})=R_{\bar U}\cM.\]
Since we have proved that $\mu_{\Sigma*}(\widetilde{R_U\cM})$ is coherent over $R_V\shD_X$, $\bar U_\bullet\cM$ is good over $V_\bullet\shD_X$.
\end{proof}

\begin{remark}\label{rmk:nabmkm}
Let $\cM$ be a holonomic $\shD_X$-module.
By Theorem \ref{thm:Lsphol}, we have $^L V_\bullet\cM$ the Kashiwara-Malgrange filtration of $\cM$ for each slope $L$. We then fix an adapted fan $\Sigma$ to a good multi-filtration $U_\bullet\cM$. Now one can naively define 
\[^\Sigma V_\bs\cM= \bigcap_{L\in \cL(\Sigma)} {}^LV_{L(\bs)}\cM\quad \forall \bs\in M,\]
which gives a $\Z^r$-filtration $^\Sigma V_\bullet \cM$ over $V_\bullet\shD_X$. However, it is not necessarily true in general that $^\Sigma V_\bullet \cM$ is good over $V_\bullet\shD_X$ even if $\cM$ is regular holonomic; see \cite[\S 3.3]{Sab} for further discussions. This means that one cannot define multi-indexed Kashiwara-Malgrange filtrations in general. On the contrary, Bernstein-Sato polynomials can be generalized successfully to the multi-indexed case (see Theorem \ref{thm:mibfs}).
\end{remark}

Using Theorem \ref{thm:goodsat}, Sabbah proved the following beautiful result about the existence of multi-variable $b$-function. We sketch its proof for completeness.
\begin{theorem}[Existence of Sabbah's generalized $b$-functions]\label{thm:mibfs}
Suppose that $\cM$ is a holonomic $\shD_X$-module with a $\Z^r$-filtration $U_\bullet\cM$ good over $V_\bullet\shD_X$ along a smooth complete intersection $Y\subseteq X$ of $H_1,\dots,H_r$. Then there exists a simplicial fan $\Sigma$ subdividing $N^+$ such that for every nonzero vector $\ba\in M^+$ there exist polynomials $b^\ba_L(s)\in\bbC[s]$ $($depending on $\ba$$)$ for all slopes $L\in \cL(\Sigma)$ so that locally
\[\prod_{L\in \cL(\Sigma)}b^\ba_L(L\cdot (t_1\partial_{t_1},\dots,t_r\partial_{t_r})) \textup{ annihilates } \frac{U_{\vec 0}\cM}{U_{-\ba}\cM},\]
where $t_j$ are local defining functions of $H_j$.
\end{theorem}
\begin{proof}
We take a simplicial fan $\Sigma$ adapted to $U_\bullet\cM$, whose existence is guaranteed by Theorem \ref{thm:existadfan}. By Theorem \ref{thm:goodsat}, the saturation of $U_\bullet\cM$ is 
\[\bar U_\bs\cM=\bigcap_{L\in \cL(\Sigma)} {}^LU_{L(\bs)}\cM \quad \forall \bs\in M.\]
Since $\bar U_\bullet\cM$ is good over $V_\bullet\shD_X$, there exists a vector $\bk\in M^+$ depending on $\ba$ such that 
\[\bar U_{-k}\cM\subseteq U_{-\ba}\cM.\]

On the other hand, similar to the proof of Theorem \ref{thm:Lsphol}, we conclude that $\gr^{^LU}_\bullet\cM$ is holonomic over $\gr^{^LV}_\bullet\shD_X$ for each slope $L$. Therefore, there exists $b_L(s)\in \bbC[s]$ such that $b_L(\theta_L)$ kills $\gr^{^LU}_\bullet\cM$. In particular, there exists $b^\ba_L(s)\in \bbC[s]$ so that 
\[b^\ba_L(L\cdot (t_1\partial_{t_1},\dots,t_r\partial_{t_r}))\cdot{ }^LU_{0}\cM\subseteq {}^LU_{L(-\bk)}\cM\]
for each $L\in \cL(\Sigma)$. Since $U_{\vec 0}\cM\subseteq \bar U_{\vec 0}\cM$ and 
\[\bar U_{-k}\cM=\bigcap_{L\in \cL(\Sigma)} {}^LU_{L(-\bk)}\cM\subseteq U_{-\ba}\cM,\]
the required statement follows. 
\end{proof}
%\begin{theorem}\label{thm:hollatticecp}
%Suppose that $\bar\cM$ is a $\shD_{X,D}$-lattice of a holonomic $\shD_X$-module with $D=\sum_{i=1}^r H_i$. Then we have a natural quasi-isomorphism
%\[\DR_{X,D}(\bar\cM(kD))\xhookrightarrow{q.i.} \DR(\cM(*D))\]
%for $k\gg 0$.\end{theorem}

\subsection{Relative characteristic cycles for $R_U\cM$}\label{subsec:relccU}
We now prove Theorem \ref{thm:CCrelRU}.
Assume that $\cM$ be a regular holonomic $\shD_X$-module with a good filtration $U_\bullet\cM$ over $V_\bullet\shD_X$. We write 
$$p\colon \varphi^{-1}((\C^*)^r)\simeq X\times (\C^*)^r\longrightarrow X$$ 
the natural projection, and $j\colon \varphi^{-1}((\C^*)^r)\hookrightarrow \wt X$ the open embedding. We then set 
$\cN=j_*(p^*\cM)$, $\cN_\rel=R_U\cM$  and $F=\prod_j u_j$. Thus, Theorem \ref{thm:CCrelRU} follows from Lemma \ref{lm:BMM}. 

If additionally $R_U\cM$ is flat over $\C^r$, then we pick an arbitrary point $\alpha\in \C^r$ and general hyperplanes $\cH_1,\dots,\cH_r$ such that $\alpha$ is a smooth complete intersection of these hyperplanes. Applying Lemma \ref{lm:BMM} and Proposition \ref{prop:spnotorsioncc} inductively, we conclude that $\CC_\rel(R_U\cM)$ is relative Lagrangian and hence that $R_U\cM$ is relative holonomic. We have thus proved Proposition \ref{prop:flatrelhol}. Now we pick a simplicial fan adapted to $U_\bullet\cM$ as in Theorem \ref{thm:existadfan}. Then $\wt{R_U\cM}$ is flat over $S_\Sigma$. By a similar argument, we can more generally prove:
\begin{prop}\label{prop:usigmarelhol}
In the situation of Theorem \ref{thm:existadfan}, if $\cM$ is a regular holonomic $\shD_X$-module, then $\wt{R_U\cM}$ is relative holonomic over $S_\Sigma$.
\end{prop}

Since the saturation $\bar U_\bullet\cM$ is good over $V_\bullet\shD_X$ (Theorem \ref{thm:goodsat}), Theorem \ref{thm:CCrelRU} specifically implies
\[\CC_\rel(R_{\bar U}\cM)=\overline{q^{-1}(\CC(\cM))}.\]
By construction, $\overline{q^{-1}(\CC(\cM))}$ is a relative conormal space but not necessarily a relative Lagrangian in general (unless $r=1$). Since 
\[R_{\bar U}\cM={\mu_\Sigma}_*(\wt{R_U\cM}),\]
from Proposition \ref{prop:usigmarelhol}, we see that the direct image functors for relative $\shD$-modules under proper base changes do not necessarily preserve relative holonomicity (cf. \S\ref{subsec:bashchangerelD}).  

\section{Graph embedding construction of Malgrange} \label{sec:gemMal}
Let $\bff=(f_1,f_2,\dots,f_r)$ be a $r$-tuple of regular (or holomorphic) functions on a smooth complex variety (or a complex manifold) $Y$. We write $$j_\bff\colon U_Y\coloneqq Y\setminus (\prod_{i=1}^rf_i=0)\hookrightarrow Y$$ the open embedding.
We consider the graph embedding 
\[\iota_{\bff}\colon Y\hookrightarrow  X\coloneqq Y\times\C^r\quad x\mapsto (x,f_1(x),\dots,f_r(x)).\]
Let $\cM_{Y}$ be a holonomic $\shD_{Y}$-module. We set $\wt\cM=\cM_Y(*D_Y)$ with the divisor $D_Y=(\prod_{i=1}^rf_i=0)$, which is a holonomic $\shD_Y$-module. We assume 
\[\wt\cM= \shD_Y\cdot \cM_0\]
for some $\sO_Y$-coherent submodule $\cM_0\subseteq \wt\cM$.
Following the idea of Malgrange \cite{MalV}, we have a coherent $\shD_Y[\bs]=\shD_Y\otimes_\C\C[\bs]$-submodule
\[\shD_Y[\bs](\bff^\bs\cdot\cM_0)\subseteq \wt\cM[\bs]\cdot\bff^\bs\]
where $\bs=(s_1,s_2,\dots,s_r)$,
\[\bff^\bs=\prod_{i=1}^rf_i^{s_i},\]
and the $\shD_Y[\bs]$-module structure is induced by 
\[\theta\cdot(\bff^\bs\cdot m_0)=\bff^\bs\cdot\theta(m_0)+\sum_{i=1}^r s_i\frac{\theta(f_i)}{f_i}\bff^\bs\cdot m_0\]
for vector fields $\theta$ on $Y$, where $m_0$ is a section of $\cM_0$. Since $\shD_Y[\bs](\bff^\bs\cdot\cM_0)$ is both a $\C[\bs]$-module and a $\shD_Y$-module, it is a coherent relative $\shD$-module over $\C[\bs]$. However,  $\wt\cM[\bs]\cdot\bff^\bs$ is not coherent over $\shD_Y[\bs]$.

We denote by $(t_1,\dots,t_r)$ the coordinates of $\C^r$. The key point is that after identifying $s_i$ with $-t_i\partial_{t_i}$, we have a $\shD_X$-module isomorphism
\[ \iota_{\bff+}(\wt\cM)\simeq \iota_{\bff*}({\wt\cM}[\bs]\bff^\bs)
\]
with the $t_i$-action on $\iota_{\bff*}({\wt\cM}[\bs]\bff^\bs)$ given by 
\[t_i\cdot (b(\bs)\bff^\bs\cdot m_0)=b(s_1,\dots,s_{i-1},s_i+1,s_{i+1},\dots,s_{r})f_i\bff^\bs\cdot m_0.\]
Consequently, $\iota_{\bff*}(\shD_Y[\bs](\bff^\bs\cdot\cM_0))$ is a $\shD_{X,D}$-lattice of $\iota_{\bff+}(\wt\cM)$, where $D$ is the divisor defined by $(t_1\cdots t_r=0)$. 
%In particular, $\shD_Y[\bs](\bff^\bs\cdot\cM_0)$ can be seen as a sheaf on $X$.
Since $\iota_{\bff*}(\shD_Y[\bs](\bff^\bs\cdot\cM_0))$ is supported on the graph of $Y$, abusing notations, we also say $\shD_Y[\bs](\bff^\bs\cdot\cM_0)$ is a $\shD_{X,D}$-lattice of $\iota_{\bff+}(\wt\cM)$. Then $\cM_0$ generates a holonomic $\shD_X$-module
\[\cM=\shD_X\cdot\iota_{\bff*}\cM_0\subseteq \iota_{\bff+}\wt\cM\]
and $\shD_Y[\bs](\bff^\bs\cdot\cM_0)$ generates a $R_V\shD_X$-module
\[R_V\shD_X\cdot\iota_{\bff*}(\shD_Y[\bs](\bff^\bs\cdot\cM_0)),\]
where the latter induces a $\Z^r$-filtration on $\cM$ with 
\[U_{\vec0}\cM=\iota_{\bff*}(\shD_Y[\bs](\bff^\bs\cdot\cM_0))\textup{ and } U_{-\vec1}\cM=\iota_{\bff*}(\shD_Y[\bs](\bff^{\bs+\vec1}\cdot\cM_0)).\]
We then apply Theorem \ref{thm:mibfs} and obtain the Sabbah's generalized $b$-function $b(\bs)\in \C[\bs]$ for $\shD_Y[\bs](\bff^\bs\cdot\cM_0)$ such that 
\[b(\bs)\cdot\dfrac{\shD_Y[\bs](\bff^\bs\cdot\cM_0)}{\shD_Y[\bs](\bff^{\bs+\vec1}\cdot\cM_0)}=0\]
with $b(\bs)$ given by a product of polynomials of degree 1. Sabbah's generalized $b$-functions associated to graph embeddings can be further generalized to notions of Bernstein-Sato ideals (see for instance \cite{Budur}).

In the graph embedding case, we can construct the log rescaled family globally (cf. \S\ref{subsec:logtorel}):
\[p: \wt X\coloneqq Y\times \C^r_t\times \C_y^r\to Y\times \C^r_t, \quad (x, t, y)\mapsto (x, e^yt).\] 
We now give a counterexample of Theorem \ref{thm:relchrellog} without flatness when $r=2$:
\begin{example}\label{ex:notrelhollogres}
We take $Y=\C^2$ with coordinates $(x_1,x_2)$ and $\bff=(x_1x_2,x_2)$. We consider the $\shD_{X,D}$-lattice $\bar\cM=\shD_Y[\bs]\bff^\bs$. Its $\shD_{X,D}$-annihilator is 
\[\Ann_{\shD_{X,D}}(\bar\cM)=(x_1\partial_{x_1}+t_1\partial_{t_1},x_2\partial_{x_2}+t_2\partial_{t_2},t_1-x_1x_2,t_2-x_2).\]
Since $\bar\cM$ and $\sM=p^*(\iota_{\bff*}\bar\cM)$
are both acyclic, 
\[\Ch_{\rel}(\sM)=(x_1\xi_{x_1}=-\xi_{y_1},x_2\xi_{x_2}=-\xi_{y_2}, e^{y_1}t_1=x_1x_2, e^{y_2}t_2=x_2).\]
Thus, the fiber of $\Ch_{\rel}(\sM)$ over $(t_1=0,t_2=0)$ satisfies
\[\Ch_{\rel}(\sM)\cap(t_1=0,t_2=0)=(x_1\xi_{x_1}=-\xi_{y_1},\xi_{y_2}=0,x_2=0)\]
and hence the dimension of the fiber is $5>4$. Therefore, $\Ch_\rel(\sM)$ is not relative Lagrangian over $\C^2_t$.
%However, $\sM$ and $\bar\cM$ are both relative holonomic over $\C[\bs]$, see \cite{BMM} for details.
\end{example}
The following theorem is a generalization of \cite[R\'esultat 1]{Mai}. See  \cite[Theorem 3.3]{WuRHA} for the proof of a more generalized result and also \cite[Theorem 4.3.4]{BVWZ2}when $\cM_0=\sO_X$. 

%The proof in \emph{loc. cit.} is involved. We sketch its proof by applying Theorem \ref{thm:mibfs}. See  \cite[\S 3]{WuRHA} for more details and also \cite[Theorem 4.3.4]{BVWZ2} for a refinement when $\wt\cM$ is regular holonomic. 
\begin{theorem}\label{thm:mairelhol}
If $\widetilde\cM$ is a holonomic $\shD_Y$-module, then every lattice $\shD_{Y}[\bs](\bff^\bs\cdot \cM_0)$
is relative holonomic over $\C[\bs]$ and 
\[\Ch^\rel(\shD_{Y}[\bs](\bff^\bs\cdot \cM_0))= \Ch(\wt\cM)\times \C^r.\]
\end{theorem}

\begin{proof}[Proof of Theorem \ref{thm:constrlattice}]
Since constructibility is local, it is enough to assume that 
\[D=(\prod_{i=1}^r x_i=0)\]
with local coordinates $(x_1,x_2,\dots,x_n)$ and $\wt\cM=\cM(*D)$ satisfies 
\[\cM(*D)=\shD_X\cdot\cM_0\]
for some coherent $\sO_X$-submodule $\cM_0$. Then we take the graph embedding of smooth log pairs
\[\iota_\bff\colon (X,D)\hookrightarrow (Z,D_Z)\]
where $\bff=(x_1,x_2,\dots,x_r)$ and $Z=X\times\C^r$, $D_Z=(t_1\cdots t_r=0)$.
Similar to the non-log case (cf. \cite[Example 1.5.23]{HTT}), we have 
\[\iota^{\log}_{\bff+}(\bar\cM)\simeq \iota_{\bff*}(\bar\cM\otimes_\sO \omega^{\log}_{\iota_\bff}\otimes_{\shD_{X,D}}f^*\shD_{Z,D_Z})\]
and 
\[\iota^{\log}_{\bff+}(\bar\cM)\hookrightarrow \iota_{\bff+}(\wt\cM).\]
Thus, $\iota^{\log}_{\bff+}(\bar\cM)$ is a $\shD_{Z,D_Z}$-lattice of the holonomic $\shD_Z$-module  $\iota_{\bff,+}(\wt\cM)$. But the graph embedding gives us different lattices of  $\iota_{\bff,+}(\wt\cM)$,
\[\shD_X[\bs](\bff\cdot\cM_0(kD))\]
for all $k\in \Z$. The lattices  $\shD_X[\bs](\bff\cdot\cM_0(kD))$ are relative holonomic over $\C[\bs]$ by Theorem \ref{thm:mairelhol}.
Meanwhile, we can compare lattices:
\[\iota^{\log}_{\bff+}(\bar\cM)\subseteq \iota_{\bff*}(\shD_X[\bs](\bff\cdot\cM_0))\otimes_{\sO_Z}\sO_Z(kD_Z)=\iota_{\bff*}(\shD_X[\bs](\bff\cdot\cM_0(kD)))\]
for $k\gg 0$. The above inclusion and Corollary \ref{cor:relholab} together imply that $\iota^{-1}_\bff(\iota^{\log}_{\bff+}(\bar\cM))$ is relative holonomic over $\C[\bs]$. By Lemma \ref{lm:spsmsubvar}, 
\[\iota^{-1}_\bff(\iota^{\log}_{\bff+}(\bar\cM))\otimes_{\C[\bs]}^\bL \C_{\mathbf 0}\simeq \bL i_{\mathbf0}^* (\iota^{-1}_\bff(\iota^{\log}_{\bff+}(\bar\cM)))\]
is a complex of $\shD_X$-modules with holonomic cohomology sheaves, where $i_{\mathbf 0}\colon \{\mathbf 0\}\hookrightarrow \Spec\C[\bs]$ is the closed embedding and $\C_{\mathbf 0}$ is the residue field. By the construction of the log de Rham complex and Proposition \ref{prop:logpushfdr}, we have
\be\label{eq:logdrtodrt}
\iota_{\bff*}(\DR_{X,D}(\bar\cM))\simeq \DR_{Z,D_Z}(\iota^{\log}_{\bff+}(\bar\cM))\simeq \iota_{\bff*}(\DR_X(\iota^{-1}_\bff(\iota^{\log}_{\bff+}(\bar\cM))\otimes_{\C[\bs]}^\bL \C_{\mathbf 0}))
\ee
where the last quasi-isomorphism follows by identifying $s_i$ with $-t_i\partial_{t_i}$. Since 
$$\iota^{-1}_\bff(\iota^{\log}_{\bff+}(\bar\cM))\otimes_{\C[\bs]}^\bL \C_{\mathbf 0}$$ is a complex of $\shD_X$-modules with holonomic cohomology sheaves, $\DR_{X,D}(\bar\cM)$ is constructible by Kashiwara's constructibility theorem (cf. \cite[Theorem 4.6.3]{HTT}).
\end{proof}
\begin{remark}\label{rmk:stratificationlogDR}
(1) From the proof of Theorem \ref{thm:constrlattice}, $\DR_{X,D}(\bar\cM)$ is not necessarily perverse in general, unless $\iota^{-1}_{\bff}(\iota_{\bff+}^{\log}(\bar\cM))$ is flat over a neighborhood of ${\mathbf0}\in \C^r$.\\
(2) The stratification for the constructible complex $\DR_{X,D}(\bar\cM)$ is determined by the stratification of $\Ch(\cM(*D))$ by \eqref{eq:logdrtodrt} and Theorem \ref{thm:mairelhol}.
\end{remark}
\begin{proof}[Proof of Theorem \ref{thm:j*j!DR}]
Part (1) is the analytification of \cite[Theorem 1.1]{WZ} with the same proof. We now prove Part (2).
We keep the notations as in the proof of Theorem \ref{thm:constrlattice}. By picking some $k\gg 0,$
we have an inclusion of lattices
\[\iota^{\log}_{\bff+}(\bar\cM)\subseteq \iota_{\bff*}(\shD_X[\bs](\bff^\bs\cdot\cM_0(kD))).\]
We then consider the short exact sequence of $\shD_{Z,D_Z}$-modules
\[0\to \iota^{\log}_{\bff+}(\bar\cM)\rightarrow\iota_{\bff*}(\shD_X[\bs](\bff\cdot\cM_0(kD)))\rightarrow\cQ\to 0,\]
where $\cQ$ is defined to be the quotient module. Applying Theorem \ref{thm:mibfs} to $\cQ$, there exists $b(\bs)\in\C[\bs]$ as a product of linear polynomials in $\bs$ such that 
\[b(\bs)\cdot \cQ=0.\]
Using substitution, we have 
\[b(\bs+\bl)\cdot\cQ(-lD_Z)=0\]
for $\bl=(l,l,\dots,l)\in \Z^r$ and for each $l$. Chose $l\gg 0$, so that $b(\bs+\bl)$ does not vanishing at ${\mathbf 0}\in \C^r$. Thus, $\cQ(-lD_Z)\otimes \C_{\mathbf 0}=0.$
Considering the above short exact sequence, since 
\[\DR_{Z,D_Z}(\cQ(-lD_Z))\simeq \iota_{\bff*}(\DR_X(\iota^{-1}_{\bff}(\cQ(-lD_Z))\otimes\C_{\mathbf 0}))\]
we hence obtain a quasi-isomorphism 
\be\label{eq:abdcde}
\DR_{Z,D_Z}(\iota^{\log}_{\bff+}(\bar\cM)(-lD_Z))\xrightarrow{q.i.} \DR_{Z,D_Z}(\iota_{\bff*}(\shD_X[\bs](\bff\cdot\cM_0((k-l)D))))
\ee
for some $l\gg k$. By construction, we have
\[\DR_{Z,D_Z}(\iota_{\bff*}(\shD_X[\bs](\bff\cdot\cM_0((k-l)D))))\simeq \iota_{\bff*}(\DR_X(\shD_X[\bs](\bff\cdot\cM_0((k-l)D))\otimes^\bL_{\C[\bs]}\C_{\mathbf 0})).\]
We then apply \cite[Corollary 5.4]{WZ}, and by the quasi-isomorphism \eqref{eq:abdcde} obtain
\[\DR_{Z,D_Z}(\iota^{\log}_{\bff+}(\bar\cM)(-lD_Z))\simeq \iota_{\bff*}(j_!(\DR(\cM|_U))).\]
By the projection formula,
\[\iota^{\log}_{\bff+}(\bar\cM)(-lD_Z)\simeq \iota^{\log}_{\bff+}(\bar\cM(-lD)).\]
Since $\iota_{\bff}$ is a closed embedding, the required quasi-isomorphism then follows from Proposition \ref{prop:logpushfdr}.

We now prove the perversity statement without the regularity assumption. Using the argument in proving \cite[Theorem 5.2 and 5.3]{WZ} as well as the discussion in \cite[\S 3.6]{BVWZ} in dealing with the local analytic case , one obtains that the $\C[\bs]$-modules $\shD_X[\bs](\bff^\bs\cdot\cM_0(kD))$ are flat over a Zariski neighborhood of ${\mathbf0}\in \Spec\C[\bs]$ for all $|k|\gg 0$. Using Sabbah's $b$-functions, we then conclude that $\iota^{-1}_\bff(\iota^{\log}_{\bff+}(\bar\cM(kD)))$ are flat over a Zariski neighborhood of ${\mathbf0}\in \Spec\C[\bs]$ for all $|k|\gg 0$. Consequently, the required perversity follows. 
\end{proof}

\bibliographystyle{amsalpha}
\bibliography{mybib}

\newcommand{\etalchar}[1]{$^{#1}$}
\providecommand{\bysame}{\leavevmode\hbox to3em{\hrulefill}\thinspace}
\providecommand{\MR}{\relax\ifhmode\unskip\space\fi MR }
\providecommand{\MRhref}[2]{%
  \href{http://www.ams.org/mathscinet-getitem?mr=#1}{#2}
}
\providecommand{\href}[2]{#2}
\begin{thebibliography}{BvdVWZ21}

\bibitem[Bj93]{Bj}
Jan-Erik Bj\"ork, \emph{Analytic {${\mathcal D}$}-modules and applications},
  Mathematics and its Applications, vol. 247, Kluwer Academic Publishers Group,
  Dordrecht, 1993. \MR{1232191}

\bibitem[BMM02]{BMM}
Jo{\"e}l Brian{\c{c}}on, Philippe Maisonobe, and Michel Merle,
  \emph{\'Equations fonctionnelles associ\'ees \`a des fonctions analytiques},
  Tr. Mat. Inst. Steklova \textbf{238} (2002), 86--96.

\bibitem[BMS06]{BMS}
Nero Budur, Mircea Musta\c{t}\v{a}, and Morihiko Saito, \emph{Bernstein-{S}ato
  polynomials of arbitrary varieties}, Compos. Math. \textbf{142} (2006),
  no.~3, 779--797. \MR{2231202}

\bibitem[BS05]{BSaito}
Nero Budur and Morihiko Saito, \emph{Multiplier ideals, {V}-filtration, and
  spectrum}, J. Algebraic Geometry \textbf{14} (2005), 269--282.

\bibitem[Bud15]{Budur}
Nero Budur, \emph{Bernstein-sato ideals and local systems}, Annales de
  l'Institut Fourier \textbf{65} (2015), no.~2, 549--603 (en).

\bibitem[BVWZ21a]{BVWZ}
Nero Budur, Robin van~der Veer, Lei Wu, and Peng Zhou, \emph{Zero loci of
  {B}ernstein--{S}ato ideals}, Invent. Math. (2021), 1--28.

\bibitem[BVWZ21b]{BVWZ2}
\bysame, \emph{Zero loci of
  {B}ernstein-{S}ato ideals {II}}, Sel. Math. New Ser.  \textbf{27}, 32 (2021).

\bibitem[CM99]{Callog}
Francisco~J Calder{\'o}n-Moreno, \emph{Logarithmic differential operators and
  logarithmic de {R}ham complexes relative to a free divisor}, no.~5, 701--714.

\bibitem[CMNM02]{Cal20}
Francisco Calder{\'o}n-Moreno and Luis Narv{\'a}ez-Macarro, \emph{The module
  {$\mathcal Df^s$} for locally quasi-homogeneous free divisors}, Compositio
  Mathematica \textbf{134} (2002), no.~1, 59--74.

\bibitem[DK16]{DAK15}
Andrea D’Agnolo and Masaki Kashiwara, \emph{Riemann-{H}ilbert correspondence
  for holonomic {D}-modules}, Publications math{\'e}matiques de l'IH{\'E}S
  \textbf{123} (2016), no.~1, 69--197.

\bibitem[ELS{\etalchar{+}}04]{ELSV}
Lawrence Ein, Robert Lazarsfeld, Karen~E. Smith, Dror Varolin, et~al.,
  \emph{Jumping coefficients of multiplier ideals}, Duke Mathematical Journal
  \textbf{123} (2004), no.~3, 469--506.

\bibitem[FFS20]{FFS19}
Luisa Fiorot, Teresa~Monteiro Fernandes, and Claude Sabbah, \emph{Relative
  regular {R}iemann-{H}ilbert correspondence}, Proceedings of the London
  Mathematical Society (2020).

\bibitem[Gin86]{Gil}
Victor Ginsburg, \emph{Characteristic varieties and vanishing cycles}, Invent.
  Math. \textbf{84} (1986), no.~2, 327--402. %\MR{833194}

\bibitem[Gin89]{Gil1}
\bysame, \emph{Admissible modules on a symmetric space}, Ast\'{e}risque (1989),
  no.~173-174, 9--10, 199--255, Orbites unipotentes et repr\'{e}sentations,
  III. \MR{1021512}

\bibitem[Gro66]{Grocm}
Alexander Grothendieck, \emph{On the de {R}ham cohomology of algebraic
  varieties}, Publications Math{\'e}matiques de l'Institut des Hautes
  {\'E}tudes Scientifiques \textbf{29} (1966), no.~1, 95--103.

\bibitem[HTT08]{HTT}
Ryoshi Hotta, Kiyoshi Takeuchi, and Toshiyuki Tanisaki, \emph{{$D$}-modules,
  perverse sheaves, and representation theory}, Progress in Mathematics, vol.
  236, Birkh\"{a}user Boston, Inc., Boston, MA, 2008, Translated from the 1995
  Japanese edition by Takeuchi. \MR{2357361}

\bibitem[Kas75]{Kascons}
Masaki Kashiwara, \emph{On the maximally overdetermined system of linear
  differential equations, {I}}, Publications of the Research Institute for
  Mathematical Sciences \textbf{10} (1975), no.~2, 563--579.

\bibitem[Kas83]{KasV}
\bysame, \emph{Vanishing cycle sheaves and holonomic systems of differential
  equations}, Algebraic geometry ({T}okyo/{K}yoto, 1982), Lecture Notes in
  Math., vol. 1016, Springer, Berlin, 1983, pp.~134--142. \MR{726425}

\bibitem[Kas03]{Kasbook}
\bysame, \emph{{$D$}-modules and microlocal calculus}, Translations of
  Mathematical Monographs, vol. 217, American Mathematical Society, Providence,
  RI, 2003, Translated from the 2000 Japanese original by Mutsumi Saito,
  Iwanami Series in Modern Mathematics. \MR{1943036}

\bibitem[Kas77]{KasBf}
\bysame, \emph{{$B$}-functions and holonomic systems. {R}ationality of roots of
  {$b$}-functions}, Invent. Math. \textbf{38} (1976/77), no.~1, 33--53.
  \MR{0430304}

\bibitem[KK79]{KK79}
Masaki Kashiwara and Takahiro Kawai, \emph{On holonomic systems for
  $\prod_{l=1}^n(f_l+\sqrt{-1}{O})^{\lambda_l}$}, Publications of the Research
  Institute for Mathematical Sciences \textbf{15} (1979), no.~2, 551--575.

\bibitem[KN99]{KN99}
Kazuya Kato and Chikara Nakayama, \emph{Log betti cohomology, log {\'e}tale
  cohomology, and log de {R}ham cohomology of log schemes over {C}}, Kodai
  Mathematical Journal \textbf{22} (1999), no.~2, 161--186.

\bibitem[Kol97]{Kolpair}
J{\'a}nos Koll{\'a}r, \emph{Singularities of pairs}, Proceedings of Symposia in
  Pure Mathematics, vol.~62, American Mathematical Society, 1997, pp.~221--288.

\bibitem[Kop20]{K20}
Clemens Koppensteiner, \emph{The de {R}ham functor for logarithmic
  {D}-modules}, Selecta Mathematica \textbf{26} (2020), no.~3, 1--44.

\bibitem[KS13]{KSbook}
Masaki Kashiwara and Pierre Schapira, \emph{Sheaves on manifolds: With a short
  history.{\guillemotleft}les d{\'e}buts de la th{\'e}orie des
  faisceaux{\guillemotright}. by christian houzel}, vol. 292, Springer Science
  \& Business Media, 2013.

\bibitem[KT19]{KT}
Clemens Koppensteiner and Mattia Talpo, \emph{Holonomic and perverse
  logarithmic {D}-modules}, Advances in Mathematics \textbf{346} (2019),
  510--545.

\bibitem[Lau83]{Laumon}
G{\'e}rard Laumon, \emph{Sur la cat\'{e}gorie d\'{e}riv\'{e}e des {${\mathcal
  D}$}-modules filtr\'{e}s}, Algebraic geometry ({T}okyo/{K}yoto, 1982),
  Lecture Notes in Math., vol. 1016, Springer, Berlin, 1983, pp.~151--237.
  \MR{726427}

\bibitem[Laz]{Laz}
Robert Lazarsfeld, \emph{Positivity in algebraic geometry. ii, volume
  49,(2004)}, Springer-Verlag \textbf{18}, no.~385, 8.

\bibitem[Mai16a]{Mai}
Philippe Maisonobe, \emph{Filtration {R}elative, l'{I}d{\'e}al de {B}ernstein
  et ses pentes}, preprint arXiv:1610.03354 (2016).

\bibitem[Mai16b]{Maihyp}
\bysame, \emph{L'{I}d{\'e}al de {B}ernstein d'un arrangement libre d'hyperplans
  lin{\'e}aires}, preprint arXiv:1610.03356 (2016).

\bibitem[Mal83]{MalV}
B.~Malgrange, \emph{Polyn\^{o}mes de {B}ernstein-{S}ato et cohomologie
  \'{e}vanescente}, Analysis and topology on singular spaces, {II}, {III}
  ({L}uminy, 1981), Ast\'{e}risque, vol. 101, Soc. Math. France, Paris, 1983,
  pp.~243--267. \MR{737934}

\bibitem[MFS16]{FS17}
Teresa Monteiro~Fernandes and Claude Sabbah, \emph{Riemann-{H}ilbert
  correspondence for mixed twistor {D}-modules}, preprint arXiv: 1609.04192
  (2016).

\bibitem[MP19a]{MP1}
Mircea Musta{\c{t}}{\u{a}} and Mihnea Popa, \emph{Hodge ideals}, Memoirs of the
  American Mathematical Society \textbf{262} (2019), no.~1268.

\bibitem[MP19b]{MP2}
\bysame, \emph{Hodge ideals for {Q}-divisors: birational approach}, Journal de
  l'{\'E}cole polytechnique \textbf{6} (2019), 283--328.

\bibitem[MP20]{MP3}
\bysame, \emph{Hodge ideals for {Q}-divisors, {V}-filtration, and minimal
  exponent}, Forum of Mathematics, Sigma, vol.~8, 2020.
  
 \bibitem[NZ09]{NZ}
David Nadler and Eric Zaslow, \emph{Constructible sheaves and the Fukaya category},  J. Amer. Math. Soc. \textbf{22} (2009).
  
 \bibitem[Ogu03]{OglogRH}
Arthur Ogus, \emph{On the logarithmic {R}iemann-{H}ilbert correspondence}, Doc.
  Math \textbf{655} (2003).

\bibitem[Sab87a]{Sab}
Claude Sabbah, \emph{Proximit\'{e} \'{e}vanescente. {I}. {L}a structure polaire
  d'un {${\mathcal D}$}-module}, Compositio Math. \textbf{62} (1987), no.~3,
  283--328. %\MR{901394}

\bibitem[Sab87b]{Sab2}
\bysame, \emph{Proximit\'{e} \'{e}vanescente. {II}. \'{E}quations
  fonctionnelles pour plusieurs fonctions analytiques}, Compositio Math.
  \textbf{64} (1987), no.~2, 213--241. %\MR{916482}

\bibitem[Sai16]{Saitohi}
Morihiko Saito, \emph{Hodge ideals and microlocal {V}-filtration}, preprint
  arXiv:1612.08667 (2016).

\bibitem[Sch12]{Schpbook}
Pierre Schapira, \emph{Microdifferential systems in the complex domain}, vol.
  269, Springer Science \& Business Media, 2012.

\bibitem[SS94]{SS94}
Pierre Schapira and Jean-Pierre Schneiders, \emph{Elliptic pairs {I}.
  {R}elative finiteness and duality}, Ast{\'e}risque (1994), no.~224, 5--60.

\bibitem[Ste88]{Ste}
Jan Stevens, \emph{On canonical singularities as total spaces of deformations},
  Abhandlungen aus dem Mathematischen Seminar der Universit{\"a}t Hamburg,
  vol.~58, Springer, 1988, p.~275.

\bibitem[Wu21]{WuRHA}
Lei Wu, \emph{Riemann-{H}ilbert correspondence for {A}lexander complexes},
preprint arXiv:2104.06941 (2021).

\bibitem[WZ21]{WZ}
Lei Wu and Peng Zhou, \emph{Log {D}-modules and index theorems}, Forum of
  Mathematics, Sigma, vol.~9, Cambridge University Press, 2021.

\end{thebibliography}

\end{document}